\documentclass[11pt]{article}
\usepackage{amssymb,amsfonts}
\usepackage{amsmath,amsthm,amsxtra}

\usepackage{graphicx}
\usepackage{amscd}
\usepackage{ulem}
\usepackage{makeidx}
\usepackage{fancybox}

\usepackage[linktocpage=true, colorlinks=true, linkcolor=blue, citecolor=red, urlcolor=green]{hyperref}

\newcommand{\ccc}{{\mathbf C}}

\newcommand{\nnn}{{\mathbf N}}

\newcommand{\rrr}{{\mathbf R}}
\newcommand{\zzz}{{\mathbf Z}}

\newtheorem{prop}{Proposition}[section]
\newtheorem{lemma}{Lemma}[section]

\newtheorem{conj}{Conjecture}[section]

\newtheorem{note}{Note}[section]

\numberwithin{equation}{section}

\begin{document}

\title{
Mock theta functions and characters of N=3 \\ superconformal modules III}

\author{\footnote{12-4 Karato-Rokkoudai, Kita-ku, Kobe 651-1334, Japan, 
\hspace{5mm}
wakimoto.minoru.314@m.kyushu-u.ac.jp, \quad wakimoto@r6.dion.ne.jp
}{ Minoru Wakimoto}}

\date{\empty}

\maketitle

\begin{center}
Abstract
\end{center}

In this paper we study the branching functions of tensor products 
of N=3 superconformal modules. 

\tableofcontents

\section{Introduction}

It is known in \cite{W2022a} and \cite{W2022b} that the characters of 
N=3 superconformal modules are written by Jacobi's $\theta$-functions and
Mumfold's $\vartheta$-functions and Dedekind's $\eta$-function. 
Also the mutiplication formula for classical theta functions is known 
in \cite{K1}.
Then it is naturally expected that these will enable us to compute 
products of N=3 characters, namely the branching functions of 
tensor product representations of the N=3 superconformal algebra. 
This paper is an approach to this problem.

For notations and definitions, we follow from \cite{W2022a}, 
\cite{W2022b} and \cite{W2022c}.

In section \ref{sec:simplest:cases}, we compute explicitly the branching functions 
of the tensor product representation $H(\Lambda^{[K(m), m_2]}) \otimes 
H(\Lambda^{[K(m'), m_2']})$ in the case when $m=m'=1$ and $m=m'=2$.

The characters of N=3 superconformal modules are obtained in \cite{W2022b}
but we need to have more information about them for our study on the 
branching of tensor product representations. 
For this sake, in section \ref{sec:n3:numerators}, 
we make further analysis on the numerators of N=3 characters.
Making use of the results obtained in section \ref{sec:n3:numerators},
we discuss about the products of N=3 characters in section 
\ref{sec:space:n3characters}.

\section{Preliminaries $\sim$ multiplication formula of theta functions}
\label{sec:preliminaries}

Multipication formula of classical theta functions in Proposition 13.2 
of \cite{K1} gives the following formula for the Jacobi's theta functions. 
\begin{lemma} 
\label{n3:lemma:2022-701a}
Let $m,n \in \nnn$ and $j, k \in \frac12 \zzz$. Then 
\begin{equation}
\theta_{j,n}(\tau,z)\theta_{k,m}(\tau,z)
=
\sum_{r \in \zzz/(m+n)\zzz}
\theta_{2mnr+kn-jm, mn(m+n)}(\tau,0) \, 
\theta_{j+k+2mr,m+n}(\tau,z)
\label{n3:eqn:2022-701a}
\end{equation}
\end{lemma}

In particular for $n=1$ and $n=2$, from this Lemma and Lemma 1.2 
in \cite{W2022c}, we obtain the following:
\begin{subequations}
\begin{equation} \left\{
\begin{array}{ccl}
\theta_{0,1}(\tau, z) \theta_{k,m}(\tau,z) 
&=& 
\sum\limits_{r \in \zzz/(m+1)\zzz} 
\theta_{k+2r,m+1}(\tau,z) \, 
\theta_{k-2mr, m(m+1)}(\tau,0) 
\\[5mm]
\theta_{1,1}(\tau, z) \theta_{k,m}(\tau,z) 
&=& 
\sum\limits_{r \in \zzz/(m+1)\zzz} 
\theta_{k+2r+1,m+1}(\tau,z) \, 
\theta_{k-(2r+1)m, m(m+1)}(\tau,0)
\end{array}\right.
\label{n3:eqn:2022-701b}
\end{equation}
\begin{equation} \left\{
\begin{array}{rcl}
\theta_{0,2}(\tau, z) \theta_{k,m}(\tau,z) 
&=& 
\sum\limits_{r \in \zzz/(m+2)\zzz} 
\theta_{k+4r,m+2}(\tau,z) \,\ \theta_{2k-4mr, 2m(m+2)}(\tau,0)
\\[5mm]
\theta_{2,2}(\tau, z) \theta_{k,m}(\tau,z) 
&=& 
\sum\limits_{r \in \zzz/(m+2)\zzz} 
\theta_{k+2+4r,m+2}(\tau,z) \,\ 
\theta_{2k-2m-4mr, 2m(m+2)}(\tau,0)
\\[5mm]
\theta_{1,2}(\tau, z) \theta_{k,m}(\tau,z) 
&=& 
\sum\limits_{r \in \zzz/(m+2)\zzz} 
\theta_{k+1+4r,m+2}(\tau,z) \,\ 
\theta_{2k-m-4mr, 2m(m+2)}(\tau,0)
\\[5mm]
\theta_{-1,2}(\tau, z) \theta_{k,m}(\tau,z) 
&=& 
\sum\limits_{r \in \zzz/(m+2)\zzz} 
\theta_{k-1+4r,m+2}(\tau,z) \,\ 
\theta_{2k+m-4mr, 2m(m+2)}(\tau,0)
\end{array} \right.
\label{n3:eqn:2022-701c}
\end{equation}
From \eqref{n3:eqn:2022-701c}, it is easy to see the following:
\begin{equation}
\big[\theta_{1,2}+\theta_{-1,2}\big](\tau,z) \, 
\theta_{k,m}(\tau,z) 
= \hspace{-3mm} 
\sum_{r \in \zzz/2(m+2)\zzz} \hspace{-3mm} 
\theta_{k-1+2r,m+2}(\tau,z) \, 
\theta_{2k+m-2mr, 2m(m+2)}(\tau,0)
\label{n3:eqn:2022-701d}
\end{equation}
\end{subequations}

\medskip

The Mumford's theta functions $\vartheta_{ab}(\tau,z)$ (cf.\cite{Mum}) 
are related to the Jacobi's theta functions by the formulas:
\begin{equation}\left\{
\begin{array}{ccc}
\vartheta_{00} &=& \theta_{0,2}+\theta_{2,2} \\[1mm]
\vartheta_{01} &=& \theta_{0,2}-\theta_{2,2}
\end{array}\right. \hspace{10mm} \left\{
\begin{array}{ccr}
\vartheta_{10} &=& \theta_{1,2}+\theta_{-1,2} \\[1mm]
\vartheta_{11} &=& i \, (\theta_{1,2}-\theta_{-1,2})
\end{array}\right.
\label{n3:eqn:2022-704b}
\end{equation}
Then, by using \eqref{n3:eqn:2022-701c}, we obtain the multiplication formulas 
for the Mumford's theta functions as follows:

\begin{lemma} \,\ 
\label{lemma:2022-616a}
\begin{enumerate}
\item[{\rm 1)}] \, 
 $\vartheta_{00}(\tau,z)\vartheta_{00}(\tau,z)
= \, 
\eta(2\tau) \bigg[\dfrac{\eta(2\tau)^2}{\eta(\tau)\eta(4\tau)}\bigg]^2 
\vartheta_{00} (2\tau, 2z)
\, + \, 
2 \, \eta(2\tau) \bigg[\dfrac{\eta(4\tau)}{\eta(2\tau)}\bigg]^2 
\vartheta_{10} (2\tau, 2z)$
\item[{\rm 2)}] \, $\vartheta_{00}(\tau,z)\vartheta_{01}(\tau,z)
\,\ = \,\ 
\eta(2\tau) \, \bigg[\dfrac{\eta(\tau)}{\eta(2\tau)}\bigg]^2 \, 
\vartheta_{01} (2\tau, 2z)$
\item[{\rm 3)}] \, $\vartheta_{01}(\tau,z)\vartheta_{01}(\tau,z)
= \, 
\eta(2\tau) \bigg[\dfrac{\eta(2\tau)^2}{\eta(\tau)\eta(4\tau)}\bigg]^2 
\vartheta_{00} (2\tau, 2z)
\, - \, 
2 \, \eta(2\tau) \bigg[\dfrac{\eta(4\tau)}{\eta(2\tau)}\bigg]^2 
\vartheta_{10} (2\tau, 2z)$
\end{enumerate}
\end{lemma}

From Lemma \ref{lemma:2022-616a}, we immediately have the following:

\begin{note} \,\ 
\label{note:2022-618b}
\begin{enumerate}
\item[{\rm 1)}] 
\begin{enumerate}
\item[{\rm (i)}] \,\ $
\dfrac{\vartheta_{10}(\tau,z)}{\vartheta_{01}(\tau,z)} \, \bigg\{
\dfrac{\eta(2\tau)^5}{\eta(\tau)^2\eta(4\tau)^2} \, \vartheta_{00}(2\tau,2z)
\, - \, 
2 \, \dfrac{\eta(4\tau)^2}{\eta(2\tau)} \, \vartheta_{10}(2\tau,2z)\bigg\}
\, = \, 
\vartheta_{01}(\tau,z)\vartheta_{10}(\tau,z)$
\item[{\rm (ii)}] \,\ $
\dfrac{\vartheta_{10}(\tau,z)}{\vartheta_{00}(\tau,z)} \, \bigg\{
\dfrac{\eta(2\tau)^5}{\eta(\tau)^2\eta(4\tau)^2} \, \vartheta_{00}(2\tau,2z)
\, + \, 
2 \, \dfrac{\eta(4\tau)^2}{\eta(2\tau)} \, \vartheta_{10}(2\tau,2z)\bigg\}
\, = \, 
\vartheta_{00}(\tau,z)\vartheta_{10}(\tau,z)$
\end{enumerate}
\item[{\rm 2)}] 
\begin{enumerate}
\item[{\rm (i)}] \,\ $
\dfrac{\vartheta_{10}(\tau,z)}{\vartheta_{01}(\tau,z)} \, 
\vartheta_{01}(2\tau,2z)
\,\ = \,\ 
\dfrac{\eta(2\tau)}{\eta(\tau)^2} \, 
\vartheta_{00}(\tau,z)\vartheta_{10}(\tau,z)$
\item[{\rm (ii)}] \,\ 
$\dfrac{\vartheta_{10}(\tau,z)}{\vartheta_{00}(\tau,z)} \, 
\vartheta_{01}(2\tau,2z)
\,\ = \,\ 
\dfrac{\eta(2\tau)}{\eta(\tau)^2} \, 
\vartheta_{01}(\tau,z)\vartheta_{10}(\tau,z)$
\end{enumerate}
\end{enumerate}
\end{note}

It is easy to show the following formulas by calculation using 
\eqref{n3:eqn:2022-701b}:

\begin{note} \,\ 
\label{n3:note:2022-704a}
\label{ex:2022-609a}
\begin{enumerate}
\item[{\rm 1)}] \,\ $\theta_{0,1}(\tau,z)^2 \,\ = \,\ 
\dfrac12 \, \eta(\tau) \, \bigg\{
\bigg[
\dfrac{\eta(\tau)^2}{\eta(\frac{\tau}{2})\eta(2\tau)}\bigg]^2 \, 
\vartheta_{00}(\tau,z)
\, + \, 
\bigg[\dfrac{\eta(\frac{\tau}{2})}{\eta(\tau)}\bigg]^2 \, 
\vartheta_{01}(\tau,z)\bigg\} $
\item[{\rm 2)}] \,\ $\theta_{1,1}(\tau,z)^2 \,\ = \,\ 
\dfrac12 \, \eta(\tau) \, \bigg\{
\bigg[
\dfrac{\eta(\tau)^2}{\eta(\frac{\tau}{2})\eta(2\tau)}\bigg]^2 \, 
\vartheta_{00}(\tau,z)
\, - \, 
\bigg[\dfrac{\eta(\frac{\tau}{2})}{\eta(\tau)}\bigg]^2 \, 
\vartheta_{01}(\tau,z)\bigg\} $

\item[{\rm 3)}] \,\ $\theta_{0,1}(\tau,z) \, \theta_{1,1}(\tau,z) 
\,\ = \,\ 
\dfrac{\eta(2\tau)^2}{\eta(\tau)} \, \vartheta_{10}(\tau,z)$
\end{enumerate}
\end{note}

Also the following formulas are obtained by easy calculation 
using Lemma 1.1 in \cite{W2022c}:

\begin{note} \,\ 
\label{note:2022-626a}
Let $m \in \nnn$ and $p \in \zzz$, then
\begin{enumerate}
\item[{\rm 1)}]
\begin{enumerate}
\item[{\rm (i)}] \,\ $
\theta_{0.m+1}\Big(\tau, \, -\dfrac12+\dfrac{m(4p+1)}{2(m+1)}\tau\Big)
\,\ = \,\ 
q^{-\frac{m^2}{m+1}(p+\frac14)^2} \, 
e^{\pi im(p+\frac14)} \, 
\theta_{m(2p+\frac12), \, m+1}(\tau, \, -\frac12)$
$$
= \,\ \left\{
\begin{array}{ccl}
q^{-\frac{m^2}{m+1}(p+\frac14)^2} \, 
\theta_{m(2p+\frac12), \, m+1}^{(-)}(\tau, \, 0)
& & {\rm if} \,\ m \, \in \, 2 \, \nnn \\[2mm]
q^{-\frac{m^2}{m+1}(p+\frac14)^2} \, 
\theta_{m(2p+\frac12), \, m+1}^{(+)}(\tau, \, 0)
& & {\rm if} \,\ m \, \in \, \nnn_{\rm odd}
\end{array}\right.
$$
\item[{\rm (ii)}] \,\ $
\theta_{0.m+1}\Big(\tau, \, -\dfrac12+\dfrac{m(4p-1)}{2(m+1)}\tau\Big)
\,\ = \,\ 
q^{-\frac{m^2}{m+1}(p-\frac14)^2} \, 
e^{\pi im(p-\frac14)} \, 
\theta_{m(2p-\frac12), \, m+1}(\tau, \, -\frac12)$
$$
= \,\ \left\{
\begin{array}{ccl}
q^{-\frac{m^2}{m+1}(p-\frac14)^2} \, 
\theta_{m(2p-\frac12), \, m+1}^{(-)}(\tau, \, 0)
& & {\rm if} \,\ m \, \in \, 2 \, \nnn \\[2mm]
q^{-\frac{m^2}{m+1}(p-\frac14)^2} \, 
\theta_{m(2p-\frac12), \, m+1}^{(+)}(\tau, \, 0)
& & {\rm if} \,\ m \, \in \, \nnn_{\rm odd}
\end{array}\right.
$$
\end{enumerate}
\item[{\rm 2)}]
\begin{enumerate}
\item[{\rm (i)}] \quad $\theta_{0.m+1}\Big(\tau, \, \dfrac{4p+1}{2(m+1)}\tau+z\Big)
\,\ = \,\ 
q^{-\frac{1}{16} \cdot \frac{(4p+1)^2}{m+1}} \, 
e^{-\frac{\pi i}{2}(4p+1)z} \, \theta_{2p+\frac12, \, m+1}(\tau, z)$
\item[{\rm (ii)}] \quad $\theta_{0.m+1}\Big(\tau, \, \dfrac{4p+1}{2(m+1)}\tau-z\Big)
\,\ = \,\ 
q^{-\frac{1}{16} \cdot \frac{(4p+1)^2}{m+1}} \, 
e^{\frac{\pi i}{2}(4p+1)z} \, \theta_{-2p-\frac12, \, m+1}(\tau, z)$
\end{enumerate}
\item[{\rm 3)}]
\begin{enumerate}
\item[{\rm (i)}] \quad $\theta_{0.m+1}\Big(\tau, \, \dfrac{4p-1}{2(m+1)}\tau+z\Big)
\,\ = \,\ 
q^{-\frac{1}{16} \cdot \frac{(4p-1)^2}{m+1}} \, 
e^{-\frac{\pi i}{2}(4p-1)z} \, \theta_{2p-\frac12, \, m+1}(\tau, z)$
\item[{\rm (ii)}] \quad $\theta_{0.m+1}\Big(\tau, \, \dfrac{4p-1}{2(m+1)}\tau-z\Big)
\,\ = \,\ 
q^{-\frac{1}{16} \cdot \frac{(4p-1)^2}{m+1}} \, 
e^{\frac{\pi i}{2}(4p-1)z} \, \theta_{-2p+\frac12, \, m+1}(\tau, z)$
\end{enumerate}
\item[{\rm 4)}]
\begin{enumerate}
\item[{\rm (i)}] \quad $\theta_{0.m+1}
\Big(\tau, \, \dfrac{4p-1}{2(m+1)}\tau+z+\tau\Big)$
$$
= \,\ 
q^{-\frac{1}{m+1}(p-\tfrac14+\tfrac{m+1}{2})^2} \, 
e^{-2\pi i (p-\frac14+\frac{m+1}{2})z} \, 
\theta_{2p-\frac12+m+1, \, m+1}(\tau, z)
$$
\item[{\rm (ii)}] \quad $\theta_{0.m+1}
\Big(\tau, \, \dfrac{4p-1}{2(m+1)}\tau-z-\tau\Big)$
$$
= \,\ 
q^{-\frac{1}{m+1}(p-\tfrac14-\tfrac{m+1}{2})^2} \, 
e^{2\pi i (p-\frac14-\frac{m+1}{2})z} \, 
\theta_{-(2p-\frac12+m+1), \, m+1}(\tau, z)
$$
\end{enumerate}
\end{enumerate}
\end{note}

\begin{note} 
\label{note:2022-626c}
For $p \in \zzz$, the following formulas hold:
\begin{enumerate}
\item[{\rm 1)}] \quad $\vartheta_{10}\Big(2\tau, \, z+\dfrac{\tau}{2}+2p\tau\Big)
\,\ = \,\ 
q^{-(p+\frac14)^2} \, e^{-2\pi i(p+\frac14)z} \, \theta_{-\frac12,1}(\tau,z)$

\vspace{1mm}

\item[{\rm 2)}] \quad $\vartheta_{10}\Big(2\tau, \, z-\dfrac{\tau}{2}-2p\tau\Big)
\,\ = \,\ 
q^{-(p+\frac14)^2} \, e^{2\pi i(p+\frac14)z} \, \theta_{\frac12,1}(\tau,z)$

\item[{\rm 3)}] \quad $\vartheta_{10}
\Big(2\tau, \, z-\dfrac{\tau}{2}-2p\tau+2\tau\Big)
\,\ = \,\ 
q^{-(p-\frac34)^2} \, e^{2\pi i(p-\frac34)z} \, \theta_{\frac12,1}(\tau,z)$
\end{enumerate}
\end{note}

From Note \ref{note:2022-626a} and Note \ref{note:2022-626c}, 
we have the following:

\begin{note} 
\label{note:2022-626d}
Let $m \in \nnn$ and $p \in \zzz$, then 
\begin{enumerate}
\item[{\rm 1)}]
\begin{enumerate}
\item[{\rm (i)}] \quad $\dfrac{\theta_{0.m+1}
\Big(\tau, \, \dfrac{4p+1}{2(m+1)}\tau+z\Big)
}{\vartheta_{10}\Big(2\tau, \, z+\dfrac{\tau}{2}+2p\tau\Big)}
\,\ = \,\ 
q^{\frac{m}{m+1}(p+\frac14)^2} \, 
\dfrac{\theta_{2p+\frac12, \, m+1}(\tau, z)}{\theta_{-\frac12,1}(\tau,z)}$
\item[{\rm (ii)}] \quad $\dfrac{\theta_{0.m+1}
\Big(\tau, \, \dfrac{4p+1}{2(m+1)}\tau-z\Big)
}{\vartheta_{10}\Big(2\tau, \, z-\dfrac{\tau}{2}-2p\tau\Big)}
\,\ = \,\ 
q^{\frac{m}{m+1}(p+\frac14)^2} \, 
\dfrac{\theta_{-2p-\frac12, \, m+1}(\tau, z)}{\theta_{\frac12,1}(\tau,z)}$
\end{enumerate}
\item[{\rm 2)}]
\begin{enumerate}
\item[{\rm (i)}] \, $\dfrac{\theta_{0.m+1}
\Big(\tau, \, \dfrac{4p-1}{2(m+1)}\tau+z+\tau\Big)
}{\vartheta_{10}\Big(2\tau, \, z+\dfrac{\tau}{2}+2p\tau\Big)}
\, = \,\ 
q^{\frac{m}{m+1}(p-\frac14)^2- \frac{m}{4}} \, 
e^{-\pi imz} \, 
\dfrac{\theta_{2p-\frac12+m+1, \, m+1}(\tau, z)}{\theta_{-\frac12,1}(\tau,z)}$
\item[{\rm (ii)}] \, $\dfrac{\theta_{0.m+1}
\Big(\tau, \, \dfrac{4p-1}{2(m+1)}\tau-z-\tau\Big)
}{\vartheta_{10}\Big(2\tau, \, z-\dfrac{\tau}{2}-2p\tau+2\tau\Big)}
\, = \,\ 
q^{\frac{m}{m+1}(p-\frac14)^2- \frac{m}{4}} \, 
e^{-\pi imz} \, 
\dfrac{\theta_{-(2p-\frac12+m+1), \, m+1}(\tau, z)}{\theta_{\frac12,1}(\tau,z)}$
\end{enumerate}
\end{enumerate}
\end{note}

\section{The simplest cases}
\label{sec:simplest:cases}

In this section we consider  N=3 modules $H(\Lambda^{[K(m), m_2]})$ when $m=1,2,4$. 
Explicit formulas of these characters are given in \cite{KW2017b} for $m=1$
and in \cite{W2022a} for $m=2$ and $m=4$.
In the setting and notations in \cite{W2022a} and \cite{W2022b}, 
these are written as follows:

\begin{prop} \,\ 
\label{prop:2022-201b}
\begin{enumerate}
\item[{\rm 1)}]
\begin{enumerate}
\item[{\rm (i)}] \quad ${\rm ch}^{(+)}_{H(\Lambda^{[K(1), 0})}(\tau, z) 
\,\ = \,\ 
- \, \dfrac{\vartheta_{00}(2\tau,z)}{\eta(\tau)}
\,\ = \,\ 
- \, \dfrac{\theta_{0,1}(\tau,z)}{\eta(\tau)}$
\item[{\rm (ii)}] \quad ${\rm ch}^{(+)}_{H(\Lambda^{[K(1), 1})}(\tau, z) 
\,\ = \,\ 
- \, i \,  \dfrac{\vartheta_{10}(2\tau,z)}{\eta(\tau)}
\,\ = \,\ 
- \, i \, \dfrac{\theta_{1,1}(\tau,z)}{\eta(\tau)}$
\end{enumerate}
\item[{\rm 2)}]
\begin{enumerate}
\item[{\rm (i)}] \,\ ${\rm ch}^{(+)}_{H(\Lambda^{[K(2), 0})}(\tau, z) 
\,\ = \,\ 
- \, \dfrac12 \, \bigg\{
\dfrac{\eta(\frac{\tau}{2})}{\eta(\tau) \, \eta(2\tau)} \, 
\vartheta_{01}(\tau,z) 
\, + \, 
\dfrac{1}{\eta(\frac{\tau}{2}) \, \eta(2\tau)} \, 
\vartheta_{00}(\tau,z)\bigg\} $
\item[{\rm (ii)}] \,\ ${\rm ch}^{(+)}_{H(\Lambda^{[K(2), 2})}(\tau, z) 
\,\ = \,\ 
- \, \dfrac12 \, \bigg\{
\dfrac{\eta(\frac{\tau}{2})}{\eta(\tau) \, \eta(2\tau)} \, 
\vartheta_{01}(\tau,z) 
\, - \, 
\dfrac{1}{\eta(\frac{\tau}{2}) \, \eta(2\tau)} \, 
\vartheta_{00}(\tau,z)\bigg\} $
\item[{\rm (iii)}] \,\ ${\rm ch}^{(+)}_{H(\Lambda^{[K(2), 1})}(\tau, z) 
\,\ = \,\ i \,\ 
\dfrac{\eta(2\tau)}{\eta(\frac{\tau}{2}) \, \eta(\tau)} \, 
\vartheta_{10}(\tau,z)$
\end{enumerate}
\item[{\rm 3)}]
\begin{enumerate}
\item[{\rm (i)}]  ${\rm ch}^{(+)}_{H(\Lambda^{[K(4),1]})}(\tau,z) 
= \, 
\dfrac{i}{2} \cdot \dfrac{1}{\eta(\frac{\tau}{2}) \, \eta(2\tau)} \, 
\bigg\{
\vartheta_{01}(\tau,z) \, \vartheta_{10}(\tau,z)
\, + \, 
\dfrac{\eta(2\tau)}{\eta(\tau)^2} \, 
\vartheta_{00}(\tau,z) \, \vartheta_{10}(\tau,z)\bigg\}$
\item[{\rm (ii)}]  ${\rm ch}^{(+)}_{H(\Lambda^{[K(4),3]})}(\tau,z) 
= \, 
\dfrac{i}{2} \cdot \dfrac{1}{\eta(\frac{\tau}{2}) \, \eta(2\tau)} \, 
\bigg\{
\vartheta_{01}(\tau,z) \, \vartheta_{10}(\tau,z)
\, - \, 
\dfrac{\eta(2\tau)}{\eta(\tau)^2} \, 
\vartheta_{00}(\tau,z) \, \vartheta_{10}(\tau,z)\bigg\}$
\end{enumerate}
\end{enumerate}
\end{prop}

\begin{proof} 1) is due to section 5 in \cite{KW2017b}, written 
in the setting and notations of \cite{W2022a}.
2) is due to Theorem 5.1 and Theorem 6.1 in \cite{W2022a}.
In order to prove 3), we rewrite the formulas for 
${\rm ch}^{(+)}_{H(\Lambda^{[K(4),m_2]})}(\tau,z)$
\,\ $(m_2=1,3)$ in Proposition 9.1 in \cite{W2022a} 
by using Note \ref{note:2022-618b} as follows:
{\allowdisplaybreaks
\begin{eqnarray*}
& & \hspace{-10mm}
{\rm ch}^{(+)}_{H(\Lambda^{[K(4), 1]})}(\tau, z)
\,\ = \,\ 
\dfrac{i}{2} \cdot \dfrac{1}{\eta(\frac{\tau}{2}) \, \eta(2\tau)} \cdot 
\dfrac{\vartheta_{10}(\tau, z)}{\vartheta_{01}(\tau, z)}
\\[1mm]
& &
\times \,\ \bigg\{
\frac{\eta(2\tau)^5}{\eta(\tau)^2 \eta(4\tau)^2} \, 
\vartheta_{00}(2\tau,2z)
\,\ - \,\ 2 \, 
\frac{\eta(4\tau)^2}{\eta(2\tau)} \, \vartheta_{10}(2\tau,2z) 
\,\ + \,\ \vartheta_{01}(2\tau,2z)\bigg\}
\\[3mm]
&=&
\dfrac{i}{2} \cdot \dfrac{1}{\eta(\frac{\tau}{2}) \, \eta(2\tau)} \, 
\Bigg\{
\underbrace{
\frac{\vartheta_{10}(\tau, z)}{\vartheta_{01}(\tau, z)} \, 
\bigg[
\frac{\eta(2\tau)^5}{\eta(\tau)^2 \eta(4\tau)^2} \, 
\vartheta_{00}(2\tau,2z)
- \, 2 \, 
\frac{\eta(4\tau)^2}{\eta(2\tau)} \, \vartheta_{10}(2\tau,2z) \bigg]
}_{\substack{|| \\[0mm] {\displaystyle 
\vartheta_{01}(\tau,z) \, \vartheta_{10}(\tau,z)
}}}
\\[-3mm]
& &
+ \,\ 
\underbrace{
\frac{\vartheta_{10}(\tau, z)}{\vartheta_{01}(\tau, z)} \, 
\vartheta_{01}(2\tau,2z)}_{\substack{|| \\[-1mm] {\displaystyle 
\frac{\eta(2\tau)}{\eta(\tau)^2} \, 
\vartheta_{00}(\tau,z) \, \vartheta_{10}(\tau,z)
}}}
\Bigg\}
\\[1mm]
&=&
\dfrac{i}{2} \cdot \dfrac{1}{\eta(\frac{\tau}{2}) \, \eta(2\tau)} \, 
\Bigg\{
\vartheta_{01}(\tau,z) \, \vartheta_{10}(\tau,z)
\, + \, 
\frac{\eta(2\tau)}{\eta(\tau)^2} \, 
\vartheta_{00}(\tau,z) \, \vartheta_{10}(\tau,z)\Bigg\}
\\[3mm]
& & \hspace{-10mm}
{\rm ch}^{(+)}_{H(\Lambda^{[K(4), 3]})}(\tau, z)
\,\ = \,\ 
\dfrac{i}{2} \cdot \dfrac{1}{\eta(\frac{\tau}{2}) \, \eta(2\tau)} \cdot 
\dfrac{\vartheta_{10}(\tau, z)}{\vartheta_{01}(\tau, z)}
\\[1mm]
& &
\times \,\ \bigg\{
\frac{\eta(2\tau)^5}{\eta(\tau)^2 \eta(4\tau)^2} \, 
\vartheta_{00}(2\tau,2z)
\,\ - \,\ 2 \, 
\frac{\eta(4\tau)^2}{\eta(2\tau)} \, \vartheta_{10}(2\tau,2z) 
\,\ - \,\ \vartheta_{01}(2\tau,2z)\bigg\}
\\[3mm]
&=&
\dfrac{i}{2} \cdot \dfrac{1}{\eta(\frac{\tau}{2}) \, \eta(2\tau)} \, 
\Bigg\{
\underbrace{
\frac{\vartheta_{10}(\tau, z)}{\vartheta_{01}(\tau, z)} \, 
\bigg[
\frac{\eta(2\tau)^5}{\eta(\tau)^2 \eta(4\tau)^2} \, 
\vartheta_{00}(2\tau,2z)
- \, 2 \, 
\frac{\eta(4\tau)^2}{\eta(2\tau)} \, \vartheta_{10}(2\tau,2z) \bigg]
}_{\substack{|| \\[0mm] {\displaystyle 
\vartheta_{01}(\tau,z) \, \vartheta_{10}(\tau,z)
}}}
\\[-3mm]
& &
- \,\ 
\underbrace{
\frac{\vartheta_{10}(\tau, z)}{\vartheta_{01}(\tau, z)} \, 
\vartheta_{01}(2\tau,2z)}_{\substack{|| \\[-1mm] {\displaystyle 
\frac{\eta(2\tau)}{\eta(\tau)^2} \, 
\vartheta_{00}(\tau,z) \, \vartheta_{10}(\tau,z)
}}}
\Bigg\}
\\[1mm]
&=&
\dfrac{i}{2} \cdot \dfrac{1}{\eta(\frac{\tau}{2}) \, \eta(2\tau)} \, 
\Bigg\{
\vartheta_{01}(\tau,z) \, \vartheta_{10}(\tau,z)
\, - \, 
\frac{\eta(2\tau)}{\eta(\tau)^2} \, 
\vartheta_{00}(\tau,z) \, \vartheta_{10}(\tau,z)\Bigg\}
\end{eqnarray*}}
proving 3).
\end{proof}

Note the following formulas obtained from this 
Proposition \ref{prop:2022-201b}:

\begin{note} \,\
\label{note:2022-609a}
\begin{enumerate}
\item[{\rm 1)}]
\begin{enumerate}
\item[{\rm (i)}] \,\ $\vartheta_{00}(\tau,z)
\, = \,  - \, 
\eta(\frac{\tau}{2})\eta(2\tau) \, \Big\{
{\rm ch}_{H(\Lambda^{[K(2), 0]})}^{(+)} (\tau,z)
\, - \, 
{\rm ch}_{H(\Lambda^{[K(2), 2]})}^{(+)} (\tau,z)\Big\}$
\item[{\rm (ii)}] \,\ $\vartheta_{01}(\tau,z)
\, = \,  - \, 
\dfrac{\eta(\tau)\eta(2\tau)}{\eta(\frac{\tau}{2})} \, \Big\{
{\rm ch}_{H(\Lambda^{[K(2), 0]})}^{(+)} (\tau,z)
\, + \, 
{\rm ch}_{H(\Lambda^{[K(2), 2]})}^{(+)} (\tau,z)\Big\}$
\end{enumerate}
\item[{\rm 2)}]
\begin{enumerate}
\item[{\rm (i)}] \,\ $\vartheta_{01}(\tau,z) \, \vartheta_{10}(\tau,z)
\, = \, 
- \, i \, \eta(\tfrac{\tau}{2}) \, \eta(2\tau) \, \big\{
{\rm ch}^{(+)}_{H(\Lambda^{[K(4),1]})}(\tau,z) 
+
{\rm ch}^{(+)}_{H(\Lambda^{[K(4),3]})}(\tau,z) \big\}$

\item[{\rm (ii)}] \,\ $\vartheta_{00}(\tau,z) \, \vartheta_{10}(\tau,z)
\, = \, 
- \, i \, \eta(\tfrac{\tau}{2}) \, \eta(\tau)^2 \, \big\{
{\rm ch}^{(+)}_{H(\Lambda^{[K(4),1]})}(\tau,z) 
-
{\rm ch}^{(+)}_{H(\Lambda^{[K(4),3]})}(\tau,z) \big\}$
\end{enumerate}
\end{enumerate}
\end{note}

Now we can prove the following propositions:

\begin{prop} 
\label{prop:2022-607a}
The characters of $H(\Lambda^{[K(1), m_2]}) \otimes H(\Lambda^{[K(1), m_2']})$ 
are as follows:
\begin{enumerate}
\item[{\rm 1)}] \quad $
{\rm ch}^{(+)}_{H(\Lambda^{[K(1), 0})}(\tau, z) \cdot
{\rm ch}^{(+)}_{H(\Lambda^{[K(1), 1})}(\tau, z) 
\,\ = \,\ 
\dfrac{\eta(\frac{\tau}{2})\eta(2\tau)}{\eta(\tau)^2} \cdot 
{\rm ch}^{(+)}_{H(\Lambda^{[K(2), 1})}(\tau, z) $
\item[{\rm 2)}] \quad ${\rm ch}^{(+)}_{H(\Lambda^{[K(1), 0})}(\tau, z)^2 
\,\ = \,\ 
- \, \dfrac12 \, \bigg\{
\dfrac{\eta(\tau)^3}{\eta(\frac{\tau}{2})\eta(2\tau)}
\, + \, 
\dfrac{\eta(\frac{\tau}{2})\eta(2\tau)}{\eta(\tau)^2} \bigg\} \, 
{\rm ch}_{H(\Lambda^{[K(2), 0]})}^{(+)} (\tau,z)$
$$ \hspace{17.3mm}
+ \,\ \frac12 \, \bigg\{
\frac{\eta(\tau)^3}{\eta(\frac{\tau}{2})\eta(2\tau)} 
\, - \, 
\frac{\eta(\frac{\tau}{2})\eta(2\tau)}{\eta(\tau)^2} \bigg\} \, 
{\rm ch}_{H(\Lambda^{[K(2), 2]})}^{(+)} (\tau,z)
$$
\item[{\rm 3)}] \quad $
{\rm ch}^{(+)}_{H(\Lambda^{[K(1), 1})}(\tau, z)^2 
\,\ = \,\ 
\dfrac12 \, \bigg\{
\dfrac{\eta(\tau)^3}{\eta(\frac{\tau}{2})\eta(2\tau)}
\, - \, 
\dfrac{\eta(\frac{\tau}{2})\eta(2\tau)}{\eta(\tau)^2} \bigg\} \, 
{\rm ch}_{H(\Lambda^{[K(2), 0]})}^{(+)} (\tau,z)$
$$ \hspace{11.5mm}
- \,\ \frac12 \, \bigg\{
\frac{\eta(\tau)^3}{\eta(\frac{\tau}{2})\eta(2\tau)} 
\, + \, 
\frac{\eta(\frac{\tau}{2})\eta(2\tau)}{\eta(\tau)^2} \bigg\} \, 
{\rm ch}_{H(\Lambda^{[K(2), 2]})}^{(+)} (\tau,z)
$$
\end{enumerate}
\end{prop}

\begin{proof} These formulas follow immediately from Proposition
\ref{prop:2022-201b} and Note \ref{n3:note:2022-704a} and 
Note \ref{note:2022-609a}.
\end{proof}

\begin{prop} 
\label{prop:2022-622a}
The characters of $H(\Lambda^{[K(2), 1]}) \otimes H(\Lambda^{[K(2), m_2]})$ 
\,\ $(m_2\in \{0,2\})$ are as follows:
\begin{enumerate}
\item[{\rm 1)}] \,\ ${\rm ch}^{(+)}_{H(\Lambda^{[K(2), 1})}(\tau,z) \cdot
{\rm ch}^{(+)}_{H(\Lambda^{[K(2), 0})}(\tau,z)$
$$ \hspace{-5mm}
= - \, \frac12 \bigg\{
\frac{\eta(\tau)}{\eta(\frac{\tau}{2})}
+
\frac{\eta(\frac{\tau}{2}) \eta(2\tau)}{\eta(\tau)^2}
\bigg\} 
{\rm ch}^{(+)}_{H(\Lambda^{[K(4),1]})}(\tau,z) 
+ \frac12 \bigg\{
\frac{\eta(\tau)}{\eta(\frac{\tau}{2})}
- 
\frac{\eta(\frac{\tau}{2}) \eta(2\tau)}{\eta(\tau)^2}\bigg\}  
{\rm ch}^{(+)}_{H(\Lambda^{[K(4),3]})}(\tau,z) 
$$
\item[{\rm 2)}] \,\ ${\rm ch}^{(+)}_{H(\Lambda^{[K(2), 1})}(\tau,z) \cdot
{\rm ch}^{(+)}_{H(\Lambda^{[K(2), 2})}(\tau,z)$
$$
= \frac12 \bigg\{
\frac{\eta(\tau)}{\eta(\frac{\tau}{2})}
-
\frac{\eta(\frac{\tau}{2}) \eta(2\tau)}{\eta(\tau)^2}
\bigg\} 
{\rm ch}^{(+)}_{H(\Lambda^{[K(4),1]})}(\tau,z) 
- \frac12 \bigg\{
\frac{\eta(\tau)}{\eta(\frac{\tau}{2})}
+ 
\frac{\eta(\frac{\tau}{2}) \eta(2\tau)}{\eta(\tau)^2}\bigg\}  
{\rm ch}^{(+)}_{H(\Lambda^{[K(4),3]})}(\tau,z) 
$$
\end{enumerate}
\end{prop}

\begin{proof} These formulas can be shown easily from Proposition
\ref{prop:2022-201b} and Note \ref{note:2022-609a} as follows:

\vspace{2mm}

\noindent
1) \quad ${\rm ch}^{(+)}_{H(\Lambda^{[K(2), 1})}(\tau,z) \cdot
{\rm ch}^{(+)}_{H(\Lambda^{[K(2), 0})}(\tau,z)$
{\allowdisplaybreaks
\begin{eqnarray*}
&= & - \frac{i}{2} \cdot 
\frac{\eta(2\tau)}{\eta(\frac{\tau}{2}) \, \eta(\tau)} \, 
\vartheta_{10}(\tau,z) \, \bigg\{
\dfrac{\eta(\frac{\tau}{2})}{\eta(\tau) \, \eta(2\tau)} \, 
\vartheta_{01}(\tau,z) 
\, + \, 
\dfrac{1}{\eta(\frac{\tau}{2}) \, \eta(2\tau)} \, 
\vartheta_{00}(\tau,z)\bigg\}
\\[3mm]
&=&
- \frac{i}{2} \cdot \frac{1}{\eta(\tau)^2} \, 
\vartheta_{10}(\tau,z) \, \vartheta_{01}(\tau,z) 
- 
\frac{i}{2} \cdot \frac{1}{\eta(\frac{\tau}{2})^2 \, \eta(\tau)} \, 
\vartheta_{10}(\tau,z) \, \vartheta_{00}(\tau,z)
\\[2mm]
&=&
- \,\ \frac{i}{2} \cdot \frac{1}{\eta(\tau)^2} \cdot 
(- \, i) \, \eta(\tfrac{\tau}{2}) \, \eta(2\tau) \, \big\{
{\rm ch}^{(+)}_{H(\Lambda^{[K(4),1]})}(\tau,z) 
+
{\rm ch}^{(+)}_{H(\Lambda^{[K(4),3]})}(\tau,z) \big\}
\\[2mm]
& &- \,\ 
\frac{i}{2} \cdot \frac{1}{\eta(\frac{\tau}{2})^2 \, \eta(\tau)} \cdot 
(- \, i) \, \eta(\tfrac{\tau}{2}) \, \eta(\tau)^2 \, \big\{
{\rm ch}^{(+)}_{H(\Lambda^{[K(4),1]})}(\tau,z) 
-
{\rm ch}^{(+)}_{H(\Lambda^{[K(4),3]})}(\tau,z) \big\}
\\[2mm]
&=&
- \frac12 \bigg\{
\frac{\eta(\frac{\tau}{2}) \eta(2\tau)}{\eta(\tau)^2}
+
\frac{\eta(\tau)}{\eta(\frac{\tau}{2})}\bigg\} 
{\rm ch}^{(+)}_{H(\Lambda^{[K(4),1]})}(\tau,z) 
+ \frac12 \bigg\{
\frac{\eta(\tau)}{\eta(\frac{\tau}{2})}
- 
\frac{\eta(\frac{\tau}{2}) \eta(2\tau)}{\eta(\tau)^2}\bigg\}  
{\rm ch}^{(+)}_{H(\Lambda^{[K(4),3]})}(\tau,z) 
\end{eqnarray*}}
proving 1).

\vspace{3mm}

\noindent
2) \quad ${\rm ch}^{(+)}_{H(\Lambda^{[K(2), 1})}(\tau,z) \cdot
{\rm ch}^{(+)}_{H(\Lambda^{[K(2), 2})}(\tau,z)$
{\allowdisplaybreaks
\begin{eqnarray*}
&= & - \frac{i}{2} \cdot 
\frac{\eta(2\tau)}{\eta(\frac{\tau}{2}) \, \eta(\tau)} \, 
\vartheta_{10}(\tau,z) \, \bigg\{
\dfrac{\eta(\frac{\tau}{2})}{\eta(\tau) \, \eta(2\tau)} \, 
\vartheta_{01}(\tau,z) 
\, - \, 
\dfrac{1}{\eta(\frac{\tau}{2}) \, \eta(2\tau)} \, 
\vartheta_{00}(\tau,z)\bigg\}
\\[3mm]
&=&
- \frac{i}{2} \cdot \frac{1}{\eta(\tau)^2} \, 
\vartheta_{10}(\tau,z) \, \vartheta_{01}(\tau,z) 
+ 
\frac{i}{2} \cdot \frac{1}{\eta(\frac{\tau}{2})^2 \, \eta(\tau)} \, 
\vartheta_{10}(\tau,z) \, \vartheta_{00}(\tau,z)
\\[2mm]
&=&
- \,\ \frac{i}{2} \cdot \frac{1}{\eta(\tau)^2} \cdot 
(- \, i) \, \eta(\tfrac{\tau}{2}) \, \eta(2\tau) \, \big\{
{\rm ch}^{(+)}_{H(\Lambda^{[K(4),1]})}(\tau,z) 
+
{\rm ch}^{(+)}_{H(\Lambda^{[K(4),3]})}(\tau,z) \big\}
\\[3mm]
& &+ \,\ 
\frac{i}{2} \cdot \frac{1}{\eta(\frac{\tau}{2})^2 \, \eta(\tau)} \cdot 
(- \, i) \, \eta(\tfrac{\tau}{2}) \, \eta(\tau)^2 \, \big\{
{\rm ch}^{(+)}_{H(\Lambda^{[K(4),1]})}(\tau,z) 
-
{\rm ch}^{(+)}_{H(\Lambda^{[K(4),3]})}(\tau,z) \big\}
\\[2mm]
&=&
\frac12 \bigg\{
\frac{\eta(\tau)}{\eta(\frac{\tau}{2})}
-
\frac{\eta(\frac{\tau}{2}) \eta(2\tau)}{\eta(\tau)^2}
\bigg\} 
{\rm ch}^{(+)}_{H(\Lambda^{[K(4),1]})}(\tau,z) 
- \frac12 \bigg\{
\frac{\eta(\tau)}{\eta(\frac{\tau}{2})}
+ 
\frac{\eta(\frac{\tau}{2}) \eta(2\tau)}{\eta(\tau)^2}\bigg\}  
{\rm ch}^{(+)}_{H(\Lambda^{[K(4),3]})}(\tau,z) 
\end{eqnarray*}}
proving 2).
\end{proof}

\section{Numerators of N=3 characters}
\label{sec:n3:numerators}

\begin{lemma} 
\label{lemma:2022-626a}
For $m \in \frac12 \nnn$ and $p \in \zzz$, the following formulas hold:
\begin{enumerate}
\item[{\rm 1)}] \,\ $
\theta_{0, 2m+1}
\Big(\tau, \, - \, \dfrac12+\dfrac{m(4p+1)\tau}{2m+1}\Big) \, 
\Phi^{[m, \frac12]}
\Big(2\tau, \,\ z+\dfrac{\tau}{2}-\dfrac12+2p\tau, \,\ 
z-\dfrac{\tau}{2}+\dfrac12-2p\tau, \,\ 0\Big)$
{\allowdisplaybreaks
\begin{eqnarray*}
&=&
-i \, \eta(2\tau)^3 \, \Bigg\{ \, 
\frac{\displaystyle 
\theta_{0, 2m+1}\Big(\tau, \, \frac{(4p+1)\tau}{2(2m+1)}+z\Big)
}{\vartheta_{10}(2\tau, \, z+\frac{\tau}{2}+2p\tau)} 
\, - \, 
\frac{\displaystyle 
\theta_{0, 2m+1}\Big(\tau, \, \frac{(4p+1)\tau}{2(2m+1)}-z\Big)
}{\vartheta_{10}(2\tau, \, z-\frac{\tau}{2}-2p\tau)} \Bigg\}
\\[2mm]
& & \hspace{-5mm}
+ \, \sum_{j=1}^{\infty} \sum_{r=1}^j 
\sum_{\substack{k=1 \\[1mm] k \, : \, {\rm odd}}}^{2m-1}
(-1)^j \, q^{j^2-\frac{1}{8m}(4mr-k)^2} %
\nonumber
\\[2mm]
& & 
\times \, \big\{
q^{(j+p+\frac14) (4mr-k)}e^{\frac{\pi i}{2}(4mr+k)}
\, + \, 
q^{(j-p-\frac14) (4mr-k)}e^{\frac{\pi i}{2}(4mr-k)}
\big\}
\big[\theta_{k,2m}-\theta_{-k,2m}](\tau, z)
\nonumber
\\[2mm]
& & \hspace{-5mm}
- \, \sum_{j=1}^{\infty} \sum_{r=0}^{j-1} 
\sum_{\substack{k=1 \\[1mm] k \, : \, {\rm odd}}}^{2m-1}
(-1)^j \, q^{j^2-\frac{1}{8m}(4mr+k)^2} 
\nonumber
\\[2mm]
& & 
\times \big\{
q^{(j+p+\frac14) (4mr+k)}e^{\frac{\pi i}{2}(4mr-k)}
\, + \, 
q^{(j-p-\frac14) (4mr+k)}e^{\frac{\pi i}{2}(4mr+k)}
\big\}
\big[\theta_{k,2m}-\theta_{-k,2m}](\tau, z)
\end{eqnarray*}}

\item[{\rm 2)}] \,\ $\theta_{0, 2m+1}
\Big(\tau, \, - \dfrac12+\dfrac{m(4p-1)\tau}{2m+1}\Big) \, 
\Phi^{[m, \frac12]}
\Big(2\tau, \, z+\dfrac{\tau}{2}-\dfrac12+2p\tau, \, 
z-\dfrac{\tau}{2}+\dfrac12-2p\tau+2\tau, \, 0\Big)$
{\allowdisplaybreaks
\begin{eqnarray*}
&=&
-i \, \eta(2\tau)^3 \, \Bigg\{ \, 
\frac{\displaystyle 
\theta_{0, 2m+1}\Big(\tau, \, \frac{(4p-1)\tau}{2(2m+1)}+z+\tau\Big)
}{\vartheta_{10}(2\tau, \, z+\frac{\tau}{2}+2p\tau)} 
\, - \, 
\frac{\displaystyle 
\theta_{0, 2m+1}\Big(\tau, \, \frac{(4p-1)\tau}{2(2m+1)}-z-\tau\Big)
}{\vartheta_{10}(2\tau, \, z-\frac{\tau}{2}-2p\tau+2\tau)} \Bigg\}
\\[2mm]
& & \hspace{-7mm}
+ \,\ q^{-\frac{m}{2}}e^{-2\pi imz}
\sum_{j=1}^{\infty} \sum_{r=1}^j 
\sum_{\substack{k=1 \\[1mm] k \, : \, {\rm odd}}}^{2m-1}
(-1)^j \, q^{j^2-\frac{1}{8m}(4mr-k)^2} %
\nonumber
\\[2mm]
& & \hspace{-5mm}
\times \big\{
q^{(j+p-\frac14) (4mr-k)}e^{\frac{\pi i}{2}(4mr+k)}
+ 
q^{(j-p+\frac14) (4mr-k)}e^{\frac{\pi i}{2}(4mr-k)}
\big\}
\big[\theta_{k+2m,2m}-\theta_{-(k+2m),2m}](\tau, z)
\nonumber
\\[2mm]
& & \hspace{-7mm}
- \,\ q^{-\frac{m}{2}}e^{-2\pi imz}
\sum_{j=1}^{\infty} \sum_{r=0}^{j-1} 
\sum_{\substack{k=1 \\[1mm] k \, : \, {\rm odd}}}^{2m-1}
(-1)^j \, q^{j^2-\frac{1}{8m}(4mr+k)^2} 
\nonumber
\\[2mm]
& & \hspace{-5mm}
\times \big\{
q^{(j+p-\frac14) (4mr+k)}e^{\frac{\pi i}{2}(4mr-k)}
+ 
q^{(j-p+\frac14) (4mr+k)}e^{\frac{\pi i}{2}(4mr+k)}
\big\}
\big[\theta_{k+2m,2m}-\theta_{-(k+2m),2m}](\tau, z)
\end{eqnarray*}}
\end{enumerate}
\end{lemma}

\begin{proof} 
This lemma follows from Corollary 3.1 in \cite{W2022b} 
by easy calculation. Namely 1) and 2) are obtained from the formula 
(3.11) in \cite{W2022b} by letting 
$(2z_1, 2z_2) =(z+\frac{\tau}{2}-\frac12+2p\tau, \, 
z-\frac{\tau}{2}+\frac12-2p\tau)$ and 
$(2z_1, 2z_2) =(z+\frac{\tau}{2}-\frac12+2p\tau, \, 
z-\frac{\tau}{2}+\frac12-2p\tau+2\tau)$ respectively.
\end{proof}

Replacing $m$ with $\frac{m}{2}$ in the above lemma, 
we obtain the following:

\begin{lemma} 
\label{lemma:2022-626b}
For $m \in \nnn$ and $p \in \zzz$, the following formulas hold:
\begin{enumerate}
\item[{\rm 1)}] \,\ $\theta_{0, m+1}
\Big(\tau, \, - \, \dfrac12+\dfrac{m(4p-1)\tau}{2(m+1)}\Big) \, 
\Phi^{[\frac{m}{2}, \frac12]}
\Big(2\tau, \, z+\dfrac{\tau}{2}-\dfrac12+2p\tau, \, 
z-\dfrac{\tau}{2}+\dfrac12-2p\tau+2\tau, \, 0\Big)$
{\allowdisplaybreaks
\begin{eqnarray*}
&=&
-i \, \eta(2\tau)^3 \, \Bigg\{ \, 
\frac{\displaystyle 
\theta_{0, m+1}\Big(\tau, \, \frac{(4p-1)\tau}{2(m+1)}+z+\tau\Big)
}{\vartheta_{10}(2\tau, \, z+\frac{\tau}{2}+2p\tau)} 
\, - \, 
\frac{\displaystyle 
\theta_{0, m+1}\Big(\tau, \, \frac{(4p-1)\tau}{2(m+1)}-z-\tau\Big)
}{\vartheta_{10}(2\tau, \, z-\frac{\tau}{2}-2p\tau+2\tau)} \Bigg\}
\\[2mm]
& & \hspace{-5mm}
+ \,\ q^{-\frac{m}{4}} e^{-\pi imz} \, 
\sum_{j=1}^{\infty} \sum_{r=1}^j 
\sum_{\substack{k=1 \\[1mm] k \, : \, {\rm odd}}}^{m-1}
(-1)^j \, q^{j^2-\frac{1}{4m}(2mr-k)^2} %
\nonumber
\\[2mm]
& & \times \big\{
q^{(j+p-\frac14) (2mr-k)}e^{\frac{\pi i}{2}(2mr+k)}
\, + \, 
q^{(j-p+\frac14) (2mr-k)}e^{\frac{\pi i}{2}(2mr-k)}
\big\}
\big[\theta_{k+m,m}-\theta_{-(k+m),m}](\tau, z)
\nonumber
\\[2mm]
& & \hspace{-5mm}
- \,\ q^{-\frac{m}{4}} e^{-\pi imz} \, 
\sum_{j=1}^{\infty} \sum_{r=0}^{j-1} 
\sum_{\substack{k=1 \\[1mm] k \, : \, {\rm odd}}}^{m-1}
(-1)^j \, q^{j^2-\frac{1}{4m}(2mr+k)^2} 
\nonumber
\\[2mm]
& & \times \big\{
q^{(j+p-\frac14) (2mr+k)}e^{\frac{\pi i}{2}(2mr-k)}
\, + \, 
q^{(j-p+\frac14) (2mr+k)}e^{\frac{\pi i}{2}(2mr+k)}
\big\}
\big[\theta_{k+m,m}-\theta_{-(k+m),m}](\tau, z)
\end{eqnarray*}}

\item[{\rm 2)}] $\theta_{0, m+1}
\Big(\tau, \, - \dfrac12+\dfrac{m(4p+1)\tau}{2(m+1)}\Big) \, 
\Phi^{[\frac{m}{2}, \frac12]}
\Big(2\tau, \, z+\dfrac{\tau}{2}-\dfrac12+2p\tau, \, 
z-\dfrac{\tau}{2}+\dfrac12-2p\tau, \, 0\Big)$
{\allowdisplaybreaks
\begin{eqnarray*}
&=&
-i \, \eta(2\tau)^3 \, \Bigg\{ \, 
\frac{\displaystyle 
\theta_{0, m+1}\Big(\tau, \, \frac{(4p+1)\tau}{2(m+1)}+z\Big)
}{\vartheta_{10}(2\tau, \, z+\frac{\tau}{2}+2p\tau)} 
\, - \, 
\frac{\displaystyle 
\theta_{0, m+1}\Big(\tau, \, \frac{(4p+1)\tau}{2(m+1)}-z\Big)
}{\vartheta_{10}(2\tau, \, z-\frac{\tau}{2}-2p\tau)} \Bigg\}
\\[2mm]
& & \hspace{-5mm}
+ \, \sum_{j=1}^{\infty} \sum_{r=1}^j 
\sum_{\substack{k=1 \\[1mm] k \, : \, {\rm odd}}}^{m-1}
(-1)^j \, q^{j^2-\frac{1}{4m}(2mr-k)^2} %
\nonumber
\\[2mm]
& & \times \, \big\{
q^{(j+p+\frac14) (2mr-k)}e^{\frac{\pi i}{2}(2mr+k)}
\, + \, 
q^{(j-p-\frac14) (2mr-k)}e^{\frac{\pi i}{2}(2mr-k)}
\big\}
\big[\theta_{k,m}-\theta_{-k,m}](\tau, z)
\nonumber
\\[2mm]
& & \hspace{-5mm}
- \, \sum_{j=1}^{\infty} \sum_{r=0}^{j-1} 
\sum_{\substack{k=1 \\[1mm] k \, : \, {\rm odd}}}^{m-1}
(-1)^j \, q^{j^2-\frac{1}{4m}(2mr+k)^2} 
\nonumber
\\[2mm]
& & \times \big\{
q^{(j+p+\frac14) (2mr+k)}e^{\frac{\pi i}{2}(2mr-k)}
\, + \, 
q^{(j-p-\frac14) (2mr+k)}e^{\frac{\pi i}{2}(2mr+k)}
\big\}
\big[\theta_{k,m}-\theta_{-k,m}](\tau, z)
\end{eqnarray*}}
\end{enumerate}
\end{lemma}

Using Note \ref{note:2022-626a} and Note \ref{note:2022-626d}, the 
formulas in the above lemma are rewritten as follows:

\begin{lemma} \quad 
\label{lemma:2022-627a}
For $m \in \nnn$ and $p \in \zzz$, the following formulas hold:
\begin{subequations}
\begin{enumerate}
\item[{\rm 1)}] \,\ $
\theta_{m(2p+\frac12), \, m+1}^{(\pm)}(\tau,0) \, 
\Phi^{[\frac{m}{2}, \frac12]}
\Big(2\tau, \,\ z+\dfrac{\tau}{2}-\dfrac12+2p\tau, \,\ 
z-\dfrac{\tau}{2}+\dfrac12-2p\tau, \,\ 0\Big)$
{\allowdisplaybreaks
\begin{eqnarray}
&=&
-i \, q^{m(p+\frac14)^2} \, 
\eta(2\tau)^3 \, \Bigg\{
\frac{\theta_{2p+\frac12, \, m+1}(\tau, z)}{\theta_{-\frac12,1}(\tau,z)}
\, - \, 
\frac{\theta_{-2p-\frac12, \, m+1}(\tau, z)}{\theta_{\frac12,1}(\tau,z)}
\Bigg\}
\nonumber
\\[2mm]
& & \hspace{-5mm}
+ \,\ q^{\frac{m^2}{m+1}(p+\frac14)^2} \, 
\sum_{j=1}^{\infty} \sum_{r=1}^j 
\sum_{\substack{k=1 \\[1mm] k \, : \, {\rm odd}}}^{m-1}
(-1)^j \, q^{j^2-\frac{1}{4m}(2mr-k)^2} %
\nonumber
\\[2mm]
& & \times \, \big\{
q^{(j+p+\frac14) (2mr-k)}e^{\frac{\pi i}{2}(2mr+k)}
\, + \, 
q^{(j-p-\frac14) (2mr-k)}e^{\frac{\pi i}{2}(2mr-k)}
\big\}
\big[\theta_{k,m}-\theta_{-k,m}](\tau, z)
\nonumber
\\[2mm]
& & \hspace{-5mm}
- \,\ q^{\frac{m^2}{m+1}(p+\frac14)^2} \, 
\sum_{j=1}^{\infty} \sum_{r=0}^{j-1} 
\sum_{\substack{k=1 \\[1mm] k \, : \, {\rm odd}}}^{m-1}
(-1)^j \, q^{j^2-\frac{1}{4m}(2mr+k)^2} 
\nonumber
\\[2mm]
& & \hspace{-10mm}
\times \big\{
q^{(j+p+\frac14) (2mr+k)}e^{\frac{\pi i}{2}(2mr-k)}
\, + \, 
q^{(j-p-\frac14) (2mr+k)}e^{\frac{\pi i}{2}(2mr+k)}
\big\}
\big[\theta_{k,m}-\theta_{-k,m}](\tau, z)
\label{eqn:2022-627a}
\end{eqnarray}}
\item[{\rm 2)}] \,\ $\theta_{m(2p-\frac12),m+1}^{(\pm)}(\tau,0) \, 
\Phi^{[\frac{m}{2}, \frac12]}
\Big(2\tau, \,\ z+\dfrac{\tau}{2}-\dfrac12+2p\tau, \,\ 
z-\dfrac{\tau}{2}+\dfrac12-2p\tau+2\tau, \,\ 0\Big)$
{\allowdisplaybreaks
\begin{eqnarray}
&=&
- \, i \,\ q^{m(p-\frac14)^2} \, q^{-\frac{m}{4}} \, 
e^{-\pi imz} \, \eta(2\tau)^3 \, 
\bigg\{
\frac{\theta_{2p-\frac12+m+1, m+1}(\tau, z)}{\theta_{-\frac12,1}(\tau, z)} 
\, - \, 
\frac{\theta_{-(2p-\frac12+m+1), m+1}(\tau, z)}{\theta_{\frac12,1}(\tau, z)}
\bigg\}
\nonumber
\\[2mm]
& & \hspace{-5mm}
+ \,\ q^{\frac{m^2}{m+1}(p-\frac14)^2} \, q^{-\frac{m}{4}} \, 
e^{-\pi imz} \, 
\sum_{j=1}^{\infty} \sum_{r=1}^j 
\sum_{\substack{k=1 \\[1mm] k \, : \, {\rm odd}}}^{m-1}
(-1)^j \, q^{j^2-\frac{1}{4m}(2mr-k)^2} %
\nonumber
\\[2mm]
& & \times \big\{
q^{(j+p-\frac14) (2mr-k)}e^{\frac{\pi i}{2}(2mr+k)}
\, + \, 
q^{(j-p+\frac14) (2mr-k)}e^{\frac{\pi i}{2}(2mr-k)}
\big\}
\big[\theta_{k+m,m}-\theta_{-(k+m),m}](\tau, z)
\nonumber
\\[2mm]
& & \hspace{-5mm}
- \,\ q^{\frac{m^2}{m+1}(p-\frac14)^2} \, q^{-\frac{m}{4}} \, 
e^{-\pi imz} \, 
\sum_{j=1}^{\infty} \sum_{r=0}^{j-1} 
\sum_{\substack{k=1 \\[1mm] k \, : \, {\rm odd}}}^{m-1}
(-1)^j \, q^{j^2-\frac{1}{4m}(2mr+k)^2} 
\nonumber
\\[2mm]
& & \times \big\{
q^{(j+p-\frac14) (2mr+k)}e^{\frac{\pi i}{2}(2mr-k)}
\, + \, 
q^{(j-p+\frac14) (2mr+k)}e^{\frac{\pi i}{2}(2mr+k)}
\big\}
\big[\theta_{k+m,m}-\theta_{-(k+m),m}](\tau, z)
\nonumber
\\[0mm]
& &
\label{eqn:2022-630a}
\end{eqnarray}}
where $\pm$ means that $+$ if $m \in \nnn_{\rm odd}$ and $-$ if $m \in \nnn_{\rm even}$.
\end{enumerate}
\end{subequations}
\end{lemma}

\begin{lemma} \,\
\label{lemma:2022-630b}
Let $m \in \nnn_{\rm odd}$ and $p \in \zzz$, then 
{\allowdisplaybreaks
\begin{eqnarray}
& & \hspace{-10mm}
\theta_{m(2p-\frac12),m+1}(\tau,0) \, 
\Phi^{[\frac{m}{2},0]}\Big(2\tau, \,\ z+\frac{\tau}{2}-\frac12+2p\tau, \,\ 
z-\frac{\tau}{2}+\frac12-2p\tau, \,\ 0\Big) 
\nonumber
\\[3mm]
&=&
- \, i \,\ e^{-\frac{\pi im}{2}} \, q^{m(p+\frac14)^2} \, 
\eta(2\tau)^3 \, \Bigg[
\frac{\theta_{2p-\frac12+m+1, m+1}(\tau, z)}{\theta_{-\frac12,1}(\tau, z)} 
\, - \, 
\frac{\theta_{-(2p-\frac12+m+1), m+1}(\tau, z)}{\theta_{\frac12,1}(\tau, z)}
\Bigg]
\nonumber
\\[2mm]
& & \hspace{-5mm}
+ \,\ 
e^{-\frac{\pi im}{2}} \, q^{mp} \, 
q^{\frac{m^2}{m+1}(p-\frac14)^2} \, 
\sum_{j=1}^{\infty} \sum_{r=1}^j 
\sum_{\substack{k=1 \\[1mm] k \, : \, {\rm odd}}}^{m-1}
(-1)^j \, q^{j^2-\frac{1}{4m}(2mr-k)^2} 
\nonumber
\\[2mm]
& & \times \big\{
q^{(j+p-\frac14) (2mr-k)}e^{\frac{\pi i}{2}(2mr+k)}
\, + \, 
q^{(j-p+\frac14) (2mr-k)}e^{\frac{\pi i}{2}(2mr-k)}
\big\}
\big[\theta_{k+m,m}-\theta_{-(k+m),m}](\tau, z)
\nonumber
\\[2mm]
& & \hspace{-5mm}
- \,\ 
e^{-\frac{\pi im}{2}} \, q^{mp} \, 
q^{\frac{m^2}{m+1}(p-\frac14)^2} \, 
\sum_{j=1}^{\infty} \sum_{r=0}^{j-1} 
\sum_{\substack{k=1 \\[1mm] k \, : \, {\rm odd}}}^{m-1}
(-1)^j \, q^{j^2-\frac{1}{4m}(2mr+k)^2} 
\nonumber
\\[2mm]
& & \times \big\{
q^{(j+p-\frac14) (2mr+k)}e^{\frac{\pi i}{2}(2mr-k)}
\, + \, 
q^{(j-p+\frac14) (2mr+k)}e^{\frac{\pi i}{2}(2mr+k)}
\big\}
\big[\theta_{k+m,m}-\theta_{-(k+m),m}](\tau, z)
\nonumber
\\[3mm]
& & \hspace{-3mm}
+ \,\ \theta_{m(2p-\frac12),m+1}(\tau,0)
\sum_{k=1}^{\frac{m-1}{2}} 
(-1)^kq^{2k(p+\frac14)} \, q^{-\frac{k^2}{m}} \, 
\big[\theta_{2k,m}- \theta_{-2k,m}\big](\tau,z)
\label{eqn:2022-630b}
\end{eqnarray}}
\end{lemma}

\begin{proof} By Lemma 3.2 in \cite{W2022b}, we have the following 
formula for $m \in \nnn_{\rm odd}$:
{\allowdisplaybreaks
\begin{eqnarray}
& & \hspace{-20mm}
\Phi^{[\frac{m}{2},0]}(2\tau, \, 2z_1, \, 2z_2, \, 0) 
\,\ = \,\ 
e^{2\pi imz_1} \, 
\Phi^{[\frac{m}{2},\frac12]}(2\tau, \, 2z_1, \, 2z_2+2\tau, \, 0) 
\nonumber
\\[2mm]
& &
+ \,\ \sum_{k=1}^{\frac{m-1}{2}} 
e^{2\pi ik(z_1-z_2)} \, q^{-\frac{k^2}{m}} \, 
\big[\theta_{k, \frac{m}{2}}-\theta_{-k, \frac{m}{2}}\big]
(2\tau, 2(z_1+z_2))
\label{n3:eqn:2022-705a}
\end{eqnarray}}

\noindent
Letting $(2z_1, 2z_2)=(z+\frac{\tau}{2}-\frac12+2p\tau, \, 
z-\frac{\tau}{2}+\frac12-2p\tau)$ in this formula 
\eqref{n3:eqn:2022-705a}, we have
{\allowdisplaybreaks
\begin{eqnarray}
& & \hspace{-10mm}
\Phi^{[\frac{m}{2},0]}\Big(2\tau, \,\ z+\frac{\tau}{2}-\frac12+2p\tau, \,\ 
z-\frac{\tau}{2}+\frac12-2p\tau, \,\ 0\Big) 
\nonumber
\\[3mm]
&=&
e^{-\frac{\pi im}{2}}q^{m(p+\frac14)}e^{\pi imz} \, 
\Phi^{[\frac{m}{2},\frac12]}\Big(2\tau, \, 
z+\frac{\tau}{2}-\frac12+2p\tau, \,\ 
z-\frac{\tau}{2}+\frac12-2p\tau+2\tau, \,\ 0\Big) 
\nonumber
\\[2mm]
& &
+ \,\ \sum_{k=1}^{\frac{m-1}{2}} 
(-1)^kq^{2k(p+\frac14)} \, q^{-\frac{k^2}{m}} \, 
\big[\theta_{k, \frac{m}{2}}-\theta_{-k, \frac{m}{2}}\big]
(2\tau, 2z)
\label{n3:eqn:2022-705b}
\end{eqnarray}}

\noindent
Multiplying \, $\theta_{m(2p-\frac12),m+1}(\tau,0)$ \, 
and substituting \eqref{eqn:2022-630a} into \eqref{n3:eqn:2022-705b},
we obtain the formula \eqref{eqn:2022-630b}, proving Lemma 
\ref{lemma:2022-630b}.
\end{proof}

\begin{lemma} 
\label{lemma:2022-627b}
Let $m \in \nnn$ and $p \in \zzz_{\geq 0}$, then 
\begin{subequations}
\begin{enumerate}
\item[{\rm 1)}] \,\ $\Phi^{[\frac{m}{2}, \frac12]}
\Big(2\tau, \,\ z+\dfrac{\tau}{2}-\dfrac12+2p\tau, \,\ 
z-\dfrac{\tau}{2}+\dfrac12-2p\tau, \,\ 0\Big)$
{\allowdisplaybreaks
\begin{eqnarray}
&=&
(-1)^{mp} \, q^{m(p+\frac14)^2-\frac{m}{16}} \, 
\Phi^{[\frac{m}{2},\frac12]}\Big(2\tau, \,\ 
z+\frac{\tau}{2}-\frac12, \,\ 
z-\frac{\tau}{2}+\frac12, \,\ 0\Big)
\nonumber
\\[2mm]
& & \hspace{-7mm}
+ \, i \, (-1)^{mp} \, q^{m(p+\frac14)^2} 
\sum_{k=1}^{pm} (-1)^k 
q^{-\frac{1}{m}(k-\frac12+\frac{m}{4})^2} 
\big[\theta_{2k-1,m}-\theta_{-(2k-1),m}\big](\tau,z)
\label{eqn:2022-627b}
\end{eqnarray}}
\item[{\rm 2)}] \,\ $\Phi^{[\frac{m}{2}, 0]}
\Big(2\tau, \,\ z+\dfrac{\tau}{2}-\dfrac12+2p\tau, \,\ 
z-\dfrac{\tau}{2}+\dfrac12-2p\tau, \,\ 0\Big)$ 
{\allowdisplaybreaks
\begin{eqnarray}
&=&
(-1)^{mp} \, q^{m(p+\frac14)^2-\frac{m}{16}} \, 
\Phi^{[\frac{m}{2},0]}\Big(2\tau, \,\ 
z+\frac{\tau}{2}-\frac12, \,\ 
z-\frac{\tau}{2}+\frac12, \,\ 0\Big)
\nonumber
\\[2mm]
& & \hspace{-7mm}
- \, (-1)^{mp} \, q^{m(p+\frac14)^2}
\sum_{k=1}^{pm} (-1)^k \, 
q^{-\frac{1}{m}(k+\frac{m}{4})^2} 
\big[\theta_{2k,m}-\theta_{-2k,m}\big](\tau,z)
\label{eqn:2022-630c}
\end{eqnarray}}
\end{enumerate}
\end{subequations}
\end{lemma}

\begin{proof} Replacing $m$ and $\tau$ with $\frac{m}{2}$ and $2\tau$ 
in the formula (2.2) in Lemma 2.4 of \cite{W2022c}, we have
{\allowdisplaybreaks
\begin{eqnarray}
& & \hspace{-15mm}
\Phi^{[\frac{m}{2},s]}(2\tau, \, z_1+2p\tau, \, z_2-2p\tau, \, 0) 
\nonumber
\\[3mm]
& & \hspace{-12mm}
= \,\ 
e^{\pi imp(z_1-z_2)} \, q^{mp^2} \, \bigg\{
\Phi^{[\frac{m}{2},s]}(2\tau, z_1,z_2,0)
\nonumber
\\[2mm]
& & \hspace{-10mm}
- \sum_{1 \, \leqq \, k \, \leqq \, pm} \hspace{-3mm}
e^{-\pi i(k-s)(z_1-z_2)} \, 
q^{- \frac{(k-s)^2}{m}} \, 
\big[\theta_{k-s, \, \frac{m}{2}} - \theta_{-(k-s), \, \frac{m}{2}}
\big] (2\tau, z_1+z_2) \bigg\}
\label{eqn:2022-702a}
\end{eqnarray}}

\noindent
Letting $(z_1,z_2) \, = \, (z+\frac{\tau}{2}-\frac12, \, 
z-\frac{\tau}{2}+\frac12)$ in this equation \eqref{eqn:2022-702a}, 
we have
{\allowdisplaybreaks
\begin{eqnarray*}
& & \hspace{-10mm}
\Phi^{[\frac{m}{2},s]}(2\tau, \, z_1+2p\tau, \, z_2-2p\tau, \, 0) 
\\[3mm]
&=&
(-1)^{mp} \, q^{m(p+\frac14)^2-\frac{m}{16}} \, 
\Phi^{[\frac{m}{2},s]}(2\tau, z_1,z_2,0)
\\[3mm]
& &
- \,\ e^{-\pi is} \,\ (-1)^{mp} \, q^{m(p+\frac14)^2} \hspace{-2mm}
\sum_{1 \, \leqq \, k \, \leqq \, pm} \hspace{-2mm}
(-1)^k \, 
q^{-\frac{1}{m}(k-s+\frac{m}{4})^2} \, 
\big[\theta_{2k-2s,m}-\theta_{-(2k-2s),m}\big](\tau,z)
\end{eqnarray*}}
Writing this formula in the cases $s=\frac12$ and $s=0$, we obtain 
formulas \eqref{eqn:2022-627b} and \eqref{eqn:2022-630c}, proving 
Lemma \ref{lemma:2022-627b}.
\end{proof}

\begin{prop} 
\label{prop:2022-628a}
Let $m \in \nnn$ and $p \in \zzz_{\geq 0}$, then
\begin{subequations}
\begin{enumerate}
\item[{\rm 1)}] \,\ $\Phi^{[\frac{m}{2},\frac12]}\Big(2\tau, \,\ 
z+\dfrac{\tau}{2}-\dfrac12, \,\ 
z-\dfrac{\tau}{2}+\dfrac12, \,\ \dfrac{\tau}{8}\Big)$
{\allowdisplaybreaks
\begin{eqnarray}
&=&
- \, i \, (-1)^{mp} \, 
\frac{\eta(2\tau)^3}{\theta_{m(2p+\frac12), \, m+1}^{(\pm)}(\tau,0)} \, 
\Bigg\{
\frac{\theta_{2p+\frac12, \, m+1}(\tau, z)}{\theta_{-\frac12,1}(\tau,z)}
\, - \, 
\frac{\theta_{-2p-\frac12, \, m+1}(\tau, z)}{\theta_{\frac12,1}(\tau,z)}
\Bigg\}
\nonumber
\\[2mm]
& & \hspace{-5mm}
+ \,\ (-1)^{mp} \, 
\frac{q^{-\frac{m}{m+1}(p+\frac14)^2}}{
\theta_{m(2p+\frac12), \, m+1}^{(\pm)}(\tau,0)} \, 
\sum_{j=1}^{\infty} \sum_{r=1}^j 
\sum_{\substack{k=1 \\[1mm] k \, : \, {\rm odd}}}^{m-1}
(-1)^j \, q^{j^2-\frac{1}{4m}(2mr-k)^2} 
\nonumber
\\[2mm]
& & \times \, \big\{
q^{(j+p+\frac14) (2mr-k)}e^{\frac{\pi i}{2}(2mr+k)}
\, + \, 
q^{(j-p-\frac14) (2mr-k)}e^{\frac{\pi i}{2}(2mr-k)}
\big\}
\big[\theta_{k,m}-\theta_{-k,m}](\tau, z)
\nonumber
\\[2mm]
& & \hspace{-5mm}
- \,\ (-1)^{mp} \, 
\frac{q^{-\frac{m}{m+1}(p+\frac14)^2}}{
\theta_{m(2p+\frac12), \, m+1}^{(\pm)}(\tau,0)} \, 
\sum_{j=1}^{\infty} \sum_{r=0}^{j-1} 
\sum_{\substack{k=1 \\[1mm] k \, : \, {\rm odd}}}^{m-1}
(-1)^j \, q^{j^2-\frac{1}{4m}(2mr+k)^2} 
\nonumber
\\[2mm]
& & \times \big\{
q^{(j+p+\frac14) (2mr+k)}e^{\frac{\pi i}{2}(2mr-k)}
\, + \, 
q^{(j-p-\frac14) (2mr+k)}e^{\frac{\pi i}{2}(2mr+k)}
\big\}
\big[\theta_{k,m}-\theta_{-k,m}](\tau, z)
\nonumber
\\[2mm]
& & \hspace{-5mm}
- \,\ i 
\sum_{1 \, \leqq \, k \, \leqq \, pm} (-1)^k \, 
q^{-\frac{1}{m}(k-\frac12+\frac{m}{4})^2} \, 
\big[\theta_{2k-1,m}-\theta_{-(2k-1),m}\big](\tau,z)
\label{eqn:2022-628a}
\end{eqnarray}}
where $\pm$ means that $+$ if $m \in \nnn_{\rm odd}$ and $-$ if $m \in \nnn_{\rm even}$.

\item[{\rm 2)}] In the case $m \in \nnn_{\rm odd}$, 
{\allowdisplaybreaks
\begin{eqnarray}
& & \hspace{-8mm}
\Phi^{[\frac{m}{2},0]}\Big(2\tau, \,\ 
z+\frac{\tau}{2}-\frac12, \,\ 
z-\frac{\tau}{2}+\frac12, \,\ \frac{\tau}{8}\Big)
\nonumber
\\[3mm]
&=& \hspace{-2mm}
- \, i \,\ (-1)^{mp} \, e^{-\frac{\pi im}{2}} \, 
\frac{\eta(2\tau)^3}{\theta_{m(2p-\frac12),m+1}(\tau,0)
} \, \Bigg[
\frac{\theta_{2p-\frac12+m+1, m+1}(\tau, z)}{\theta_{-\frac12,1}(\tau, z)} 
\, - \, 
\frac{\theta_{-(2p-\frac12+m+1), m+1}(\tau, z)}{\theta_{\frac12,1}(\tau, z)}
\Bigg]
\nonumber
\\[2mm]
& & \hspace{-7mm}
+ \,\ (-1)^{mp} \, e^{-\frac{\pi im}{2}} \, 
\frac{q^{-\frac{m}{m+1}(p-\frac14)^2}}{\theta_{m(2p-\frac12),m+1}(\tau,0)} \, 
\sum_{j=1}^{\infty} \sum_{r=1}^j 
\sum_{\substack{k=1 \\[1mm] k \, : \, {\rm odd}}}^{m-1}
(-1)^j \, q^{j^2-\frac{1}{4m}(2mr-k)^2} 
\nonumber
\\[2mm]
& & \hspace{-4mm}
\times \big\{
q^{(j+p-\frac14) (2mr-k)}e^{\frac{\pi i}{2}(2mr+k)}
\, + \, 
q^{(j-p+\frac14) (2mr-k)}e^{\frac{\pi i}{2}(2mr-k)}
\big\}
\big[\theta_{k+m,m}-\theta_{-(k+m),m}](\tau, z)
\nonumber
\\[2mm]
& & \hspace{-7mm}
- \,\ (-1)^{mp} \, e^{-\frac{\pi im}{2}} \, 
\frac{q^{-\frac{m}{m+1}(p-\frac14)^2}}{\theta_{m(2p-\frac12),m+1}(\tau,0)} \, 
\sum_{j=1}^{\infty} \sum_{r=0}^{j-1} 
\sum_{\substack{k=1 \\[1mm] k \, : \, {\rm odd}}}^{m-1}
(-1)^j \, q^{j^2-\frac{1}{4m}(2mr+k)^2} 
\nonumber
\\[2mm]
& & \hspace{-4mm}
\times \big\{
q^{(j+p-\frac14) (2mr+k)}e^{\frac{\pi i}{2}(2mr-k)}
\, + \, 
q^{(j-p+\frac14) (2mr+k)}e^{\frac{\pi i}{2}(2mr+k)}
\big\}
\big[\theta_{k+m,m}-\theta_{-(k+m),m}](\tau, z)
\nonumber
\\[3mm]
& & \hspace{-7mm}
+ \,\ (-1)^{mp} \, 
\sum_{k=1}^{\frac{m-1}{2}} (-1)^k \, 
q^{-m(p+\frac14-\frac{k}{m})^2} \, 
\big[\theta_{2k,m}- \theta_{-2k,m}\big](\tau,z)
\nonumber
\\[3mm]
& & \hspace{-7mm}
+ \,\ \sum_{1 \, \leqq \, k \, \leqq \, pm} 
(-1)^k \, 
q^{-\frac{1}{m}(k+\frac{m}{4})^2} \, 
\big[\theta_{2k,m}-\theta_{-2k,m}\big](\tau,z)
\label{eqn:2022-630d}
\end{eqnarray}}
\end{enumerate}
\end{subequations}
\end{prop}

\begin{proof} 1) \,\ Substituting \eqref{eqn:2022-627b} into 
\eqref{eqn:2022-627a}, we have 
{\allowdisplaybreaks
\begin{eqnarray*}
& & \hspace{-7mm}
(-1)^{mp} \, q^{m(p+\frac14)^2} \, 
\theta_{m(2p+\frac12), \, m+1}^{(\pm)}(\tau,0) \, \bigg\{
q^{-\frac{m}{16}} \, 
\Phi^{[\frac{m}{2},\frac12]}\Big(2\tau, \,\ 
z+\frac{\tau}{2}-\frac12, \,\ 
z-\frac{\tau}{2}+\frac12, \,\ 0\Big)
\nonumber
\\[3mm]
& & + \,\ i 
\sum_{1 \, \leqq \, k \, \leqq \, pm} 
(-1)^k \, 
q^{-\frac{1}{m}(k-\frac12+\frac{m}{4})^2} \, 
\big[\theta_{2k-1,m}-\theta_{-(2k-1),m}\big](\tau,z)
\bigg\}
\nonumber
\\[3mm]
&=&
- \, i \, q^{m(p+\frac14)^2} \, 
\eta(2\tau)^3 \, \Bigg\{
\frac{\theta_{2p+\frac12, \, m+1}(\tau, z)}{\theta_{-\frac12,1}(\tau,z)}
\, - \, 
\frac{\theta_{-2p-\frac12, \, m+1}(\tau, z)}{\theta_{\frac12,1}(\tau,z)}
\Bigg\}
\nonumber
\\[2mm]
& & \hspace{-5mm}
+ \,\ q^{\frac{m^2}{m+1}(p+\frac14)^2} \, 
\sum_{j=1}^{\infty} \sum_{r=1}^j 
\sum_{\substack{k=1 \\[1mm] k \, : \, {\rm odd}}}^{m-1}
(-1)^j \, q^{j^2-\frac{1}{4m}(2mr-k)^2} %
\nonumber
\\[2mm]
& & \times \, \big\{
q^{(j+p+\frac14) (2mr-k)}e^{\frac{\pi i}{2}(2mr+k)}
\, + \, 
q^{(j-p-\frac14) (2mr-k)}e^{\frac{\pi i}{2}(2mr-k)}
\big\}
\big[\theta_{k,m}-\theta_{-k,m}](\tau, z)
\nonumber
\\[2mm]
& & \hspace{-5mm}
- \,\ q^{\frac{m^2}{m+1}(p+\frac14)^2} \, 
\sum_{j=1}^{\infty} \sum_{r=0}^{j-1} 
\sum_{\substack{k=1 \\[1mm] k \, : \, {\rm odd}}}^{m-1}
(-1)^j \, q^{j^2-\frac{1}{4m}(2mr+k)^2} 
\nonumber
\\[2mm]
& & \times \big\{
q^{(j+p+\frac14) (2mr+k)}e^{\frac{\pi i}{2}(2mr-k)}
\, + \, 
q^{(j-p-\frac14) (2mr+k)}e^{\frac{\pi i}{2}(2mr+k)}
\big\}
\big[\theta_{k,m}-\theta_{-k,m}](\tau, z)
\end{eqnarray*}}

\noindent
Dividing both sides of this equation by 
$(-1)^{mp} \, q^{m(p+\frac14)^2}
\theta_{m(2p+\frac12), \, m+1}^{(\pm)}(\tau,0)$
and noticing that 
$$
q^{-\frac{m}{16}}
\Phi^{[\frac{m}{2},\frac12]}\Big(2\tau, \,\ 
z+\frac{\tau}{2}-\frac12, \,\ 
z-\frac{\tau}{2}+\frac12, \,\ 0\Big)
=
\Phi^{[\frac{m}{2},\frac12]}\Big(2\tau, \,\ 
z+\frac{\tau}{2}-\frac12, \,\ 
z-\frac{\tau}{2}+\frac12, \,\ \frac{\tau}{8}\Big) \, ,
$$
we obtain the formula \eqref{eqn:2022-628a}.

\vspace{3mm}

\noindent
2) \,\ Substituting \eqref{eqn:2022-630c} into 
\eqref{eqn:2022-630a}, we have 
{\allowdisplaybreaks
\begin{eqnarray*}
& & \hspace{-8mm}
(-1)^{mp} \, q^{m(p+\frac14)^2} \, 
\theta_{m(2p-\frac12),m+1}(\tau,0) \, \bigg\{
q^{-\frac{m}{16}} \, 
\Phi^{[\frac{m}{2},0]}\Big(2\tau, \,\ 
z+\frac{\tau}{2}-\frac12, \,\ 
z-\frac{\tau}{2}+\frac12, \,\ 0\Big)
\\[3mm]
& & 
- \,\ 
\sum_{1 \, \leqq \, k \, \leqq \, pm} \hspace{-2mm}
(-1)^k \, 
q^{-\frac{1}{m}(k+\frac{m}{4})^2} \, 
\big[\theta_{2k,m}-\theta_{-2k,m}\big](\tau,z)
\bigg\}
\\[3mm]
&=&
- \, i \,\ e^{-\frac{\pi im}{2}} \, q^{m(p+\frac14)^2} \, 
\eta(2\tau)^3 \, \Bigg[
\frac{\theta_{2p-\frac12+m+1, m+1}(\tau, z)}{\theta_{-\frac12,1}(\tau, z)} 
\, - \, 
\frac{\theta_{-(2p-\frac12+m+1), m+1}(\tau, z)}{\theta_{\frac12,1}(\tau, z)}
\Bigg]
\\[2mm]
& & \hspace{-5mm}
+ \,\ 
e^{-\frac{\pi im}{2}} \, q^{mp} \, 
q^{\frac{m^2}{m+1}(p-\frac14)^2} \, 
\sum_{j=1}^{\infty} \sum_{r=1}^j 
\sum_{\substack{k=1 \\[1mm] k \, : \, {\rm odd}}}^{m-1}
(-1)^j \, q^{j^2-\frac{1}{4m}(2mr-k)^2} 
\\[2mm]
& & \times \big\{
q^{(j+p-\frac14) (2mr-k)}e^{\frac{\pi i}{2}(2mr+k)}
\, + \, 
q^{(j-p+\frac14) (2mr-k)}e^{\frac{\pi i}{2}(2mr-k)}
\big\}
\big[\theta_{k+m,m}-\theta_{-(k+m),m}](\tau, z)
\\[2mm]
& & \hspace{-5mm}
- \,\ 
e^{-\frac{\pi im}{2}} \, q^{mp} \, 
q^{\frac{m^2}{m+1}(p-\frac14)^2} \, 
\sum_{j=1}^{\infty} \sum_{r=0}^{j-1} 
\sum_{\substack{k=1 \\[1mm] k \, : \, {\rm odd}}}^{m-1}
(-1)^j \, q^{j^2-\frac{1}{4m}(2mr+k)^2} 
\\[2mm]
& & \times \big\{
q^{(j+p-\frac14) (2mr+k)}e^{\frac{\pi i}{2}(2mr-k)}
\, + \, 
q^{(j-p+\frac14) (2mr+k)}e^{\frac{\pi i}{2}(2mr+k)}
\big\}
\big[\theta_{k+m,m}-\theta_{-(k+m),m}](\tau, z)
\\[3mm]
& & \hspace{-3mm}
+ \,\ \theta_{m(2p-\frac12),m+1}(\tau,0)
\sum_{k=1}^{\frac{m-1}{2}} 
(-1)^kq^{2k(p+\frac14)} \, q^{-\frac{k^2}{m}} \, 
\big[\theta_{2k,m}- \theta_{-2k,m}\big](\tau,z)
\end{eqnarray*}}

\noindent
Dividing both sides of this equation by 
$(-1)^{mp} \, q^{m(p+\frac14)^2}
\theta_{m(2p-\frac12), \, m+1}(\tau,0)$, 
we obtain the formula \eqref{eqn:2022-630d}.
\end{proof}

\section{The space of N=3 characters}
\label{sec:space:n3characters}

We put
$$
\begin{array}{lcl}
\ccc [[q^{\frac12}]] &:=& \ccc\text{-linear span of} \,\ \bigg\{
q^b \, \sum\limits_{j=0}^{\infty} \, a_j \, (q^{\frac12})^j 
\quad ; \quad 
b \, \in \, \rrr, \,\ a_j \, \in \, \ccc
\bigg\}
\\[4mm]
\ccc ((q^{\frac12})) &:=& \bigg\{ \dfrac{f}{g} \quad ; \quad 
f, \, g \, \in \, \ccc [[q^{\frac12}]], \,\ 
g \, \ne \, 0 \bigg\}
\end{array}
$$
and
{\allowdisplaybreaks
\begin{eqnarray*}
& & \hspace{-15mm}
V^{[m,\frac12]} :=
\ccc((q^{\frac12}))\text{-linear span of} \,\ 
\Big\{ \big(\overset{N=3}{R}{}^{(+)} \, 
{\rm ch}^{(+)}_{H(\Lambda^{[K(m), m_2]})}\big)(\tau,z)
\,\ ; \,\ m_2 \, \in \, 2\zzz\Big\} 
\\[2mm]
&=&
\ccc((q^{\frac12}))\text{-}{\rm linear \,\ span \,\ of}  
\\[2mm]
& & \Bigg\{
\Phi^{[\frac{m}{2};s]}\Big(2\tau, \, z+\frac{\tau}{2}-\frac12, \, 
z-\frac{\tau}{2}+\frac12, \, \frac{\tau}{8}\Big) \,\ ; \,\ 
s \, \in \, \tfrac12 \zzz_{\rm odd}, \,\ 
\frac12 \leqq s \leqq \frac{m+1}{2}\Bigg\}
\\[3mm]
& & \hspace{-15mm}
V^{[m,0]} := 
\ccc((q^{\frac12}))\text{-linear span of} \,\ 
\Big\{ \big(\overset{N=3}{R}{}^{(+)} \, 
{\rm ch}^{(+)}_{H(\Lambda^{[K(m), m_2]})}\big)(\tau,z)
\,\ ; \,\ m_2 \, \in \, \nnn_{\rm odd} \Big\} 
\\[2mm]
&=&
\ccc((q^{\frac12}))\text{-}{\rm linear \,\ span \,\ of}  
\\[2mm]
& & \Bigg\{
\Phi^{[\frac{m}{2};s]}\Big(2\tau, \, z+\frac{\tau}{2}-\frac12, \, 
z-\frac{\tau}{2}+\frac12, \, \frac{\tau}{8}\Big) \,\ ; \,\ 
s \, \in \, \zzz, \,\ 
1 \leqq s \leqq \frac{m+1}{2}\Bigg\}
\end{eqnarray*}}
for $m \in \nnn$. 
We also define $U^{[m,s]}$, for $m \in \nnn$ and $s \in \frac12 \zzz$,
as follows:
\begin{subequations}
{\allowdisplaybreaks
\begin{eqnarray}
& & \hspace{-10mm}
U^{[m,s]} \quad (s \in \tfrac12+\zzz)
\nonumber
\\[0mm]
&:=&
\ccc((q^{\frac12}))\text{-linear span of} \,\ 
\Bigg\{
\frac{\theta_{\frac12, m+1}(\tau,z)}{\theta_{-\frac12,1}(\tau,z)}
-
\frac{\theta_{-\frac12, m+1}(\tau,z)}{\theta_{\frac12,1}(\tau,z)}, 
\quad 
\underset{ \Big(\substack{\\[1mm]
k \, \in \, \nnn_{\rm odd} \\[1mm]
1 \, \leqq \, k \, \leqq \, m-1
} \Big) }{\big[\theta_{k,m}-\theta_{-k,m}\big](\tau,z)}
\Bigg\}
\nonumber
\\[-6mm]
& &
\label{n3:eqn:2022-708a1}
\\[0mm]
& & \hspace{-10mm}
U^{[m,s]} \quad (s \in \zzz)
\nonumber
\\[0mm]
&:=&
\ccc((q^{\frac12}))\text{-linear span of} \,\ 
\Bigg\{
\frac{\theta_{m+\frac12, m+1}(\tau,z)}{\theta_{-\frac12,1}(\tau,z)}
-
\frac{\theta_{-(m+\frac12), m+1}(\tau,z)}{\theta_{\frac12,1}(\tau,z)}, 
\quad 
\underset{ \Big(\substack{\\[1mm]
k \, \in \, \nnn_{\rm even} \\[1mm]
1 \, \leqq \, k \, \leqq \, m-1
} \Big) }{\big[\theta_{k,m}-\theta_{-k,m}\big](\tau,z)}
\Bigg\}
\nonumber
\\[-2mm]
& &
\label{n3:eqn:2022-708a2}
\end{eqnarray}}
\end{subequations}
Then \hspace{10mm}
$U^{[m,s]} \,\ = \,\ \left\{
\begin{array}{lcl}
U^{[m,\frac12]} & & {\rm if} \quad s \, \in \, \frac12+\zzz \\[1mm]
U^{[m,0]} & & {\rm if} \quad s \, \in \, \zzz
\end{array}\right. \, . $

\vspace{2mm}

The following formula follows immediately from Lemma 2.1 in \cite{W2022b}:

\begin{note} 
\label{note:2022-702a}
Let $m \in \nnn$ and $s \in \frac12 \zzz$, then
{\allowdisplaybreaks
\begin{eqnarray}
& & \hspace{-20mm}
\Phi^{[\frac{m}{2}, s]}\Big(2\tau, \, z+\dfrac{\tau}{2}-\dfrac12, \, 
z-\dfrac{\tau}{2}+\dfrac12, \, \dfrac{\tau}{8}\Big)
-
\Phi^{[\frac{m}{2}, s+1]}\Big(2\tau, \, z+\dfrac{\tau}{2}-\dfrac12, \, 
z-\dfrac{\tau}{2}+\dfrac12, \, \dfrac{\tau}{8}\Big)
\nonumber
\\[2.5mm]
&=&
e^{-\pi is} \, q^{-\frac{1}{m}(s-\frac{m}{4})^2} \, 
\big[\theta_{2s, \, m}-\theta_{-2s, \, m}\big](\tau, z)
\label{n3:eqn:2022-704a}
\end{eqnarray}}
\end{note}

\noindent
Then, by the above Note \ref{note:2022-702a}, the spaces $V^{[m,s]}$  
\,\ $(s=\frac12, 0)$ \, can be written as follows:

\begin{note} \,\ 
\label{n3:note:2022-706a}
Let $m \in \nnn$, then
\begin{enumerate}
\item[{\rm 1)}] \,\ $V^{[m,\frac12]}
\,\ = \,\ 
\ccc((q^{\frac12}))$-linear span of
$$
\Bigg\{
\Phi^{[\frac{m}{2};\frac12]}\Big(2\tau, \, z+\frac{\tau}{2}-\frac12, \, 
z-\frac{\tau}{2}+\frac12, \, \frac{\tau}{8}\Big), \quad \big[
\underset{\Big(\substack{
k \, \in \, \nnn_{\rm odd} \\[1mm] 
1 \leqq k \leqq m-1
}\Big)
}{\theta_{k,m}-\theta_{-k,m}}
\big](\tau,z)\Bigg\}
$$
\item[{\rm 2)}] \,\ $V^{[m,0]}
\,\ = \,\ 
\ccc((q^{\frac12}))$-linear span of
$$
\Bigg\{
\Phi^{[\frac{m}{2};0]}\Big(2\tau, \, z+\frac{\tau}{2}-\frac12, \, 
z-\frac{\tau}{2}+\frac12, \, \frac{\tau}{8}\Big), \quad \big[
\underset{\Big(\substack{
k \, \in \, 2\nnn \\[1mm] 2 \, \leqq \, k \, \leqq \, m-1
}\Big)
}{\theta_{k,m}-\theta_{-k,m}}
\big](\tau,z)\Bigg\}
$$
\end{enumerate}
\end{note}

As to the spaces $U^{[m,s]}$, we note the following:

\begin{lemma} \,\ 
\label{lemma:2022-702a}
\begin{enumerate}
\item[{\rm 1)}] \,\ $
\dfrac{\theta_{2p+\frac12,m+1}(\tau,z)}{\theta_{-\frac12,1}(\tau,z)}
\, - \, 
\dfrac{\theta_{-(2p+\frac12),m+1}(\tau,z)}{\theta_{\frac12,1}(\tau,z)}
\,\ \in \,\ U^{[m,\frac12]}$  \quad for \,\ $m \in \nnn$ and $p \in \zzz$.
\item[{\rm 2)}] \,\ $
\dfrac{\theta_{2p+\frac12+m,m+1}(\tau,z)}{\theta_{-\frac12,1}(\tau,z)}
\, - \, 
\dfrac{\theta_{-(2p+\frac12+m),m+1}(\tau,z)}{\theta_{\frac12,1}(\tau,z)}
\,\ \in \,\ U^{[m,0]}$  \quad for \,\ $m \in \nnn_{\rm odd}$ and $p \in \zzz$.
\end{enumerate}
\end{lemma}

\begin{proof} Claims 1) and 2) hold when $p \in \zzz_{\geq 0}$, 
due to Proposition \ref{prop:2022-628a} since the LHS's of 
\eqref{eqn:2022-628a} and \eqref{eqn:2022-630d} do not depend on $p$. 
So we need to prove them only in the case when $p \in \zzz_{<0}$.
We choose $a \in \zzz$ such that $p+a(m+1)\geq 0$ and put 
$p' := p+a(m+1) \,\ (\in \zzz_{\geq 0})$. 
Then we have

\medskip

\noindent
1) \quad $
\dfrac{\theta_{2p+\frac12,m+1}(\tau,z)}{\theta_{-\frac12,1}(\tau,z)}
\, - \, 
\dfrac{\theta_{-(2p+\frac12),m+1}(\tau,z)}{\theta_{\frac12,1}(\tau,z)}$

{\allowdisplaybreaks
\begin{eqnarray*}
&=&
\frac{\theta_{2p+2a(m+1)+\frac12,m+1}(\tau,z)}{\theta_{-\frac12,1}(\tau,z)}
\, - \, 
\frac{\theta_{-2p-2a(m+1)-\frac12,m+1}(\tau,z)}{\theta_{\frac12,1}(\tau,z)}
\\[3mm]
&=&
\dfrac{\theta_{2p'+\frac12,m+1}(\tau,z)}{\theta_{-\frac12,1}(\tau,z)}
\, - \, 
\dfrac{\theta_{-(2p'+\frac12),m+1}(\tau,z)}{\theta_{\frac12,1}(\tau,z)}
\,\ \in \,\ U^{[m, \frac12]}
\end{eqnarray*}}

\noindent
2) \quad $
\dfrac{\theta_{2p+\frac12+m,m+1}(\tau,z)}{\theta_{-\frac12,1}(\tau,z)}
\, - \, 
\dfrac{\theta_{-(2p+\frac12+m),m+1}(\tau,z)}{\theta_{\frac12,1}(\tau,z)}$

{\allowdisplaybreaks
\begin{eqnarray*}
&=&
\frac{\theta_{2p+2a(m+1)+\frac12+m,m+1}(\tau,z)}{\theta_{-\frac12,1}(\tau,z)}
\, - \, 
\frac{\theta_{-2p-2a(m+1)-\frac12-m,m+1}(\tau,z)}{\theta_{\frac12,1}(\tau,z)}
\\[3mm]
&=&
\dfrac{\theta_{2p'+\frac12+m,m+1}(\tau,z)}{\theta_{-\frac12,1}(\tau,z)}
\, - \, 
\dfrac{\theta_{-(2p'+\frac12+m),m+1}(\tau,z)}{\theta_{\frac12,1}(\tau,z)}
\,\ \in \,\ U^{[m, 0]}
\end{eqnarray*}}
Thus the proof of Lemma \ref{lemma:2022-702a} is completed.
\end{proof}

In the case $m \in \nnn_{\rm odd}$, the space $U^{[m,0]}$ 
has a bit simpler characterization:

\begin{lemma} \,\
\label{n3:lemma:2022-707a}
In the case $m \in \nnn_{\rm odd}$, 
\begin{enumerate}
\item[{\rm 1)}] \,\ $U^{[m,0]} \,\ = \,\ \ccc((q^{\frac12}))$-linear span of
$$
\Bigg\{
\frac{\theta_{-\frac12, m+1}(\tau,z)}{\theta_{-\frac12,1}(\tau,z)}
\, - \, 
\frac{\theta_{\frac12, m+1}(\tau,z)}{\theta_{\frac12,1}(\tau,z)}, 
\quad 
\underset{ \Big(\substack{\\[1mm]
k \, \in \, \nnn_{\rm even} \\[1mm]
1 \, \leqq \, k \, \leqq \, m-1
} \Big) }{\big[\theta_{k,m}-\theta_{-k,m}\big](\tau,z)}
\Bigg\}
$$
\item[{\rm 2)}] \,\ $
\dfrac{\theta_{2p-\frac12,m+1}(\tau,z)}{\theta_{-\frac12,1}(\tau,z)}
\, - \, 
\dfrac{\theta_{-(2p-\frac12),m+1}(\tau,z)}{\theta_{\frac12,1}(\tau,z)}
\,\ \in \,\ U^{[m,0]}$  \quad for \,\ $p \in \zzz$.
\end{enumerate}
\end{lemma}

\begin{proof} In the case $m \in \nnn_{\rm odd}$, letting $p=-\frac{m+1}{2}$ in 2) 
of Lemma \ref{lemma:2022-702a}, we have 
$$
\frac{\theta_{-\frac12, m+1}(\tau,z)}{\theta_{-\frac12,1}(\tau,z)}
\, - \, 
\frac{\theta_{\frac12, m+1}(\tau,z)}{\theta_{\frac12,1}(\tau,z)}
\,\ \in \, U^{[m,0]} \, ,
$$
proving 1). \, 2) is obtained from 2) of Lemma \ref{lemma:2022-702a}
by replacing $p$ with $p - \frac{m+1}{2}$.
\end{proof}

The spaces $U^{[m,s]}$ and $V^{[m,s]}$ are related as follows:

\begin{lemma} \,\
\label{n3:lemma:2022-706a}
\begin{enumerate}
\item[{\rm 1)}] \,\ $V^{[m,0]} \,\ = \,\ U^{[m,0]}$ \hspace{10mm}
for \quad $m \in \nnn_{\rm odd}$.
\item[{\rm 2)}] \,\ $V^{[m,\frac12]} \,\ = \,\ U^{[m,\frac12]}$ \hspace{10mm}
for \quad $m \in \nnn$.
\end{enumerate}
\end{lemma}

\begin{proof} This lemma follows immediately from 
Proposition \ref{prop:2022-628a} and Lemma \ref{lemma:2022-702a} 
and Note \ref{note:2022-702a} and Note \ref{n3:note:2022-706a}.
\end{proof}

We now consider the multiplication of theta functions with 
functions in $U^{[m,s]}$. First we note the following simple fact:

\begin{note} \,\ 
\label{n3:note:2022-708a}
Let $m \in \nnn$ and $n \in \{1, 2\}$ and $j, k \in \zzz$. Then 
$$
\theta_{k,m}-\theta_{-k,m} \, \in \, U^{[m, \, s]}
\quad \Longrightarrow \quad 
\theta_{j,n} \cdot \big[\theta_{k,m}-\theta_{-k,m}\big]
\, \in \, U^{[m+n, \, s+\frac{j}{2}]}
$$
\end{note}

\begin{proof} First consider the case $n=1$, Then, 
by \eqref{n3:eqn:2022-701b}, we have 
\begin{subequations}
{\allowdisplaybreaks
\begin{eqnarray}
& & \hspace{-20mm}
\theta_{0,1}(\tau,z) \, \big[\theta_{k,m}-\theta_{-k,m}\big](\tau, z)
\nonumber
\\[0mm]
&=&
\sum_{r=0}^m \, 
\theta_{k-2rm, \, m(m+1)}(\tau,0) \cdot 
\big[\theta_{k+2r, \, m+1}-\theta_{-(k+2r), \, m+1}\big](\tau,z)
\label{eqn:2022-702b1}
\\[1mm]
& & \hspace{-20mm}
\theta_{1,1}(\tau,z) \, \big[\theta_{k,m}-\theta_{-k,m}\big](\tau, z)
\nonumber
\\[0mm]
&=&
\sum_{r=0}^m \, 
\theta_{k-(2r+1)m, \, m(m+1)}(\tau,0) \cdot 
\big[\theta_{k+2r+1, \, m+1}-\theta_{-(k+2r+1), \, m+1}\big](\tau,z)
\label{eqn:2022-702b2}
\end{eqnarray}}
\end{subequations}

\noindent
From these equations and the definition of $U^{[m,s]}$, we see that
{\allowdisplaybreaks
\begin{eqnarray*}
\theta_{k,m}-\theta_{-k,m} \in U^{[m,\frac12]} \hspace{5mm} 
({\rm i.e,} \,\ k \in \nnn_{\rm odd})
&\Longrightarrow& \left\{
\begin{array}{ccl}
\text{RHS of \eqref{eqn:2022-702b1}} &\in & U^{[m+1, \frac12]} \\[1mm]
\text{RHS of \eqref{eqn:2022-702b2}} &\in & U^{[m+1, 0]}
\end{array} \right.
\\[1mm]
\theta_{k,m}-\theta_{-k,m} \in U^{[m,0]} \hspace{5mm} 
({\rm i.e,} \,\ k \in \nnn_{\rm even})
&\Longrightarrow& \left\{
\begin{array}{ccl}
\text{RHS of \eqref{eqn:2022-702b1}} &\in & U^{[m+1, 0]} \\[1mm]
\text{RHS of \eqref{eqn:2022-702b2}} &\in & U^{[m+1, \frac12]}
\end{array} \right.
\end{eqnarray*}}
namely
\begin{subequations}
{\allowdisplaybreaks
\begin{eqnarray}
\theta_{k,m}-\theta_{-k,m} \in U^{[m,\frac12]} 
&\Longrightarrow& \left\{
\begin{array}{ccl}
\theta_{0,1} \cdot \big[\theta_{k,m}-\theta_{-k,m}\big] &\in & U^{[m+1, \frac12]} \\[1mm]
\theta_{1,1} \cdot \big[\theta_{k,m}-\theta_{-k,m}\big] &\in & U^{[m+1, 0]}
\end{array} \right.
\label{n3:eqn:2022-706c1}
\\[1mm]
\theta_{k,m}-\theta_{-k,m} \in U^{[m,0]} 
&\Longrightarrow& \left\{
\begin{array}{ccl}
\theta_{0,1} \cdot \big[\theta_{k,m}-\theta_{-k,m}\big] &\in & U^{[m+1, 0]} \\[1mm]
\theta_{1,1} \cdot \big[\theta_{k,m}-\theta_{-k,m}\big] &\in & U^{[m+1, \frac12]}
\end{array} \right.
\label{n3:eqn:2022-706c2}
\end{eqnarray}}
\end{subequations}
proving Note \ref{n3:note:2022-708a} in the case $n=1$. The proof in the 
case $n=2$ is quite similar.
\end{proof}

By easy calculation using \eqref{n3:eqn:2022-701b} and 
\eqref{n3:eqn:2022-701c} and \eqref{n3:eqn:2022-701d}, one obtains 
the following:

\begin{note} 
\label{note:2022-707a}
For $m \, \in \, \nnn$, the following formulas hold:
\begin{enumerate}
\item[{\rm 1)}]
\begin{enumerate}
\item[{\rm (i)}] \quad $\theta_{0,1}(\tau,z) \, 
\Bigg[
\dfrac{\theta_{\frac12, m+1}(\tau,z)}{\theta_{-\frac12, 1}(\tau,z)}
-
\dfrac{\theta_{-\frac12, m+1}(\tau,z)}{\theta_{\frac12, 1}(\tau,z)}\Bigg]$
$$
=
\sum_{r \in \zzz/(m+2)\zzz}
\theta_{\frac12-2r(m+1), (m+1)(m+2)}(\tau,0) \, \Bigg[
\frac{\theta_{\frac12+2r,m+2}(\tau,z)}{\theta_{-\frac12, 1}(\tau,z)} 
-
\frac{\theta_{-\frac12-2r,m+2}(\tau,z)}{\theta_{\frac12, 1}(\tau,z)}\Bigg]
$$
\item[{\rm (ii)}] \quad $\theta_{0,1}(\tau,z) \, 
\Bigg[
\underbrace{
\dfrac{\theta_{-\frac12, m+1}(\tau,z)}{\theta_{-\frac12, 1}(\tau,z)}
-
\dfrac{\theta_{\frac12, m+1}(\tau,z)}{\theta_{\frac12, 1}(\tau,z)}}_{
\substack{\rotatebox{-90}{$\in$} \\[-0.5mm] {\displaystyle 
U^{[m,0]} \,\ {\rm if} \,\ m \in \nnn_{\rm odd}
}}}
\Bigg]$
$$
=
\sum_{r \in \zzz/(m+2)\zzz}
\theta_{\frac12+2r(m+1), (m+1)(m+2)}(\tau,0) \, \Bigg[
\underbrace{
\frac{\theta_{-\frac12+2r,m+2}(\tau,z)}{\theta_{-\frac12, 1}(\tau,z)} 
-
\frac{\theta_{\frac12-2r,m+2}(\tau,z)}{\theta_{\frac12, 1}(\tau,z)}
}_{\substack{\rotatebox{-90}{$\in$} \\[-0.5mm] {\displaystyle 
U^{[m+1,0]} \,\ {\rm if} \,\ m \in \nnn_{\rm even}
}}}
\Bigg]
$$
\end{enumerate}
\item[{\rm 2)}]
\begin{enumerate}
\item[{\rm (i)}] \quad $\theta_{1,1}(\tau,z) \, \Bigg[
\dfrac{\theta_{\frac12, m+1}(\tau,z)}{\theta_{-\frac12, 1}(\tau,z)}
-
\dfrac{\theta_{-\frac12, m+1}(\tau,z)}{\theta_{\frac12, 1}(\tau,z)}\Bigg]$
$$
=
\sum_{r \in \zzz/(m+2)\zzz}
\theta_{\frac12-(2r-1)(m+1), (m+1)(m+2)}(\tau,0) \, \Bigg[
\underbrace{
\frac{\theta_{-\frac12+2r,m+2}(\tau,z)}{\theta_{-\frac12, 1}(\tau,z)} 
-
\frac{\theta_{\frac12-2r,m+2}(\tau,z)}{\theta_{\frac12, 1}(\tau,z)}
}_{\substack{\rotatebox{-90}{$\in$} \\[-0.5mm] {\displaystyle 
U^{[m+1,0]} \,\ {\rm if} \,\ m \in \nnn_{\rm even}
}}}
\Bigg]
$$
\item[{\rm (ii)}] \quad $\theta_{1,1}(\tau,z) \, \Bigg[
\underbrace{
\dfrac{\theta_{-\frac12, m+1}(\tau,z)}{\theta_{-\frac12, 1}(\tau,z)}
-
\dfrac{\theta_{\frac12, m+1}(\tau,z)}{\theta_{\frac12, 1}(\tau,z)}
}_{\substack{\rotatebox{-90}{$\in$} \\[-0.5mm] {\displaystyle 
U^{[m,0]} \,\ {\rm if} \,\ m \in \nnn_{\rm odd}
}}}
\Bigg]$
$$
= \sum_{r \in \zzz/(m+2)\zzz}
\theta_{\frac12+(2r+1)(m+1), (m+1)(m+2)}(\tau,0) \, \Bigg[
\frac{\theta_{\frac12+2r,m+2}(\tau,z)}{\theta_{-\frac12, 1}(\tau,z)} 
-
\frac{\theta_{-\frac12-2r,m+2}(\tau,z)}{\theta_{\frac12, 1}(\tau,z)}\Bigg]
$$
\end{enumerate}
\item[{\rm 3)}]
\begin{enumerate}
\item[{\rm (i)}] \quad $\theta_{0,2}(\tau,z) \, \Bigg[
\dfrac{\theta_{\frac12, m+1}(\tau,z)}{\theta_{-\frac12,1}(\tau,z)} 
-
\dfrac{\theta_{-\frac12, m+1}(\tau,z)}{\theta_{\frac12,1}(\tau,z)}\Bigg]$
$$
=\sum_{r \in \zzz/(m+3)\zzz}
\theta_{1-4r(m+1), 2(m+1)(m+3)}(\tau,0) \, \Bigg[
\frac{\theta_{\frac12+4r,m+3}(\tau,z)}{\theta_{-\frac12,1}(\tau,z)} 
-
\frac{\theta_{-(\frac12+4r),m+3}(\tau,z)}{\theta_{\frac12,1}(\tau,z)}\Bigg]
$$
\item[{\rm (ii)}] \quad $\theta_{0,2}(\tau,z) \, \Bigg[
\underbrace{
\dfrac{\theta_{-\frac12, m+1}(\tau,z)}{\theta_{-\frac12,1}(\tau,z)} 
-
\dfrac{\theta_{\frac12, m+1}(\tau,z)}{\theta_{\frac12,1}(\tau,z)}
}_{\substack{\rotatebox{-90}{$\in$} \\[-0.5mm] {\displaystyle 
U^{[m,0]} \,\ {\rm if} \,\ m \in \nnn_{\rm odd}
}}}
\Bigg]$
$$
=\sum_{r \in \zzz/(m+3)\zzz}
\theta_{1+4r(m+1), 2(m+1)(m+3)}(\tau,0) \, \Bigg[
\underbrace{
\frac{\theta_{-\frac12+4r,m+3}(\tau,z)}{\theta_{-\frac12,1}(\tau,z)} 
-
\frac{\theta_{\frac12-4r,m+3}(\tau,z)}{\theta_{\frac12,1}(\tau,z)}
}_{\substack{\rotatebox{-90}{$\in$} \\[-0.5mm] {\displaystyle 
U^{[m+2,0]} \,\ {\rm if} \,\ m \in \nnn_{\rm odd}
}}}
\Bigg]
$$
\end{enumerate}
\item[{\rm 4)}]
\begin{enumerate}
\item[{\rm (i)}] \quad $\theta_{2,2}(\tau,z) \, \Bigg[
\dfrac{\theta_{\frac12, m+1}(\tau,z)}{\theta_{-\frac12,1}(\tau,z)} 
-
\dfrac{\theta_{-\frac12, m+1}(\tau,z)}{\theta_{\frac12,1}(\tau,z)}\Bigg]$

$$ \hspace{-10mm}
= \sum_{r \in \zzz/(m+3)\zzz} \hspace{-3mm}
\theta_{1-2(2r+1)(m+1), 2(m+1)(m+3)}(\tau,0) \, \Bigg[
\frac{\theta_{\frac12+2(2r+1),m+3}(\tau,z)}{\theta_{-\frac12,1}(\tau,z)} 
-
\frac{\theta_{-\frac12-2(2r+1)),m+3}(\tau,z)}{\theta_{\frac12,1}(\tau,z)}\Bigg]
$$
\item[{\rm (ii)}] \quad $\theta_{2,2}(\tau,z) \, \Bigg[
\underbrace{
\dfrac{\theta_{-\frac12, m+1}(\tau,z)}{\theta_{-\frac12,1}(\tau,z)} 
-
\dfrac{\theta_{\frac12, m+1}(\tau,z)}{\theta_{\frac12,1}(\tau,z)}
}_{\substack{\rotatebox{-90}{$\in$} \\[-0.5mm] {\displaystyle 
U^{[m,0]} \,\ {\rm if} \,\ m \in \nnn_{\rm odd}
}}}
\Bigg]$
$$ \hspace{-10mm}
= \sum_{r \in \zzz/(m+3)\zzz} \hspace{-3mm}
\theta_{1+2(2r+1)(m+1), 2(m+1)(m+3)}(\tau,0) \, \Bigg[
\underbrace{
\frac{\theta_{-\frac12+2(2r+1),m+3}(\tau,z)}{\theta_{-\frac12,1}(\tau,z)} 
-
\frac{\theta_{\frac12-2(4r+1),m+3}(\tau,z)}{\theta_{\frac12,1}(\tau,z)}
}_{\substack{\rotatebox{-90}{$\in$} \\[-0.5mm] {\displaystyle 
U^{[m+2,0]} \,\ {\rm if} \,\ m \in \nnn_{\rm odd}
}}}
\Bigg]
$$
\end{enumerate}
\item[{\rm 5)}]
\begin{enumerate}
\item[{\rm (i)}] \quad $\vartheta_{10}(\tau,z) \, \Bigg[
\dfrac{\theta_{\frac12, m+1}(\tau,z)}{\theta_{-\frac12,1}(\tau,z)} 
-
\dfrac{\theta_{-\frac12, m+1}(\tau,z)}{\theta_{\frac12,1}(\tau,z)}\Bigg]$
$$
=\sum_{r \in \zzz/2(m+3)\zzz} \hspace{-5mm}
\theta_{-1+(1-2r)(m+1), 2(m+1)(m+3)}(\tau,0) \, \Bigg[
\underbrace{
\frac{\theta_{-\frac12+2r,m+3}(\tau,z)}{\theta_{-\frac12,1}(\tau,z)} 
-
\frac{\theta_{\frac12-2r,m+3}(\tau,z)}{\theta_{\frac12,1}(\tau,z)} 
}_{\substack{\rotatebox{-90}{$\in$} \\[-0.5mm] {\displaystyle 
U^{[m+2,0]} \,\ {\rm if} \,\ m \in \nnn_{\rm odd}
}}}
\Bigg]
$$
\item[{\rm (ii)}] \quad $\vartheta_{10}(\tau,z) \, \Bigg[
\underbrace{
\dfrac{\theta_{-\frac12, m+1}(\tau,z)}{\theta_{-\frac12,1}(\tau,z)} 
-
\dfrac{\theta_{\frac12, m+1}(\tau,z)}{\theta_{\frac12,1}(\tau,z)}
}_{\substack{\rotatebox{-90}{$\in$} \\[-0.5mm] {\displaystyle 
U^{[m,0]} \,\ {\rm if} \,\ m \in \nnn_{\rm odd}
}}}
\Bigg]$
$$
=
\sum_{r \in \zzz/2(m+3)\zzz} \hspace{-5mm}
\theta_{-1+(1+2r)(m+1), 2(m+1)(m+3)}(\tau,0) \, \Bigg[
\frac{\theta_{\frac12+2r,m+3}(\tau,z)}{\theta_{-\frac12,1}(\tau,z)} 
-
\frac{\theta_{-\frac12-2r,m+3}(\tau,z)}{\theta_{\frac12,1}(\tau,z)} \Bigg]
$$
\end{enumerate}
\end{enumerate}
\end{note}

Then the multiplication of theta functions with functions in 
$U^{[m, s]}$ is obtained as follows:

\begin{lemma} \quad 
\label{lemma:2022-708a}
\begin{enumerate}
\item[{\rm 1)}] \quad $\theta_{0,1} \cdot U^{[m, \frac12]}
\,\ \subset \,\ U^{[m+1, \frac12]}$
\item[{\rm 2)}]
\begin{enumerate}
\item[{\rm (i)}] \quad $\theta_{1,1} \cdot U^{[m, \frac12]}
\,\ \subset \,\ U^{[m+1, 0]}$ \quad {\rm if} \quad 
$m \, \in \, \nnn_{\rm even}$
\item[{\rm (ii)}] \quad $\theta_{1,1} \cdot U^{[m, 0]}
\,\ \subset \,\ U^{[m+1, \frac12]}$ \quad {\rm if} \quad 
$m \, \in \, \nnn_{\rm odd}$
\end{enumerate}
\item[{\rm 3)}]
\begin{enumerate}
\item[{\rm (i)}] \quad $\theta_{j,2} \cdot U^{[m, \frac12]}
\,\ \subset \,\ U^{[m+2, \frac12]}$ \quad $(j \in \{0,2\})$
\item[{\rm (ii)}] \quad $\theta_{j,2} \cdot U^{[m, 0]}
\,\ \subset \,\ U^{[m+2, 0]}$ \quad $(j \in \{0,2\})$ 
\quad {\rm if} \quad 
$m \, \in \, \nnn_{\rm odd}$
\end{enumerate}
\item[{\rm 4)}]
\begin{enumerate}
\item[{\rm (i)}] \quad $\vartheta_{10} \cdot U^{[m, \frac12]}
\,\ \subset \,\ U^{[m+2, 0]}$ \quad {\rm if} \quad 
$m \, \in \, \nnn_{\rm odd}$
\item[{\rm (ii)}] \quad $\vartheta_{10} \cdot U^{[m, 0]}
\,\ \subset \,\ U^{[m+2, \frac12]}$ \quad {\rm if} \quad 
$m \, \in \, \nnn_{\rm odd}$
\end{enumerate}
\item[{\rm 5)}]
\begin{enumerate}
\item[{\rm (i)}] \quad $\vartheta_{10} \, \vartheta_{0b} 
\cdot U^{[m, \frac12]}
\,\ \subset \,\ U^{[m+4, 0]}$ \,\ $(b\in \{0,1\})$ \,\ 
\quad {\rm if} \quad $m \, \in \, \nnn_{\rm odd}$
\item[{\rm (ii)}] \quad $\vartheta_{10} \, \vartheta_{0b} 
\cdot U^{[m, 0]}
\,\ \subset \,\ U^{[m+4, \frac12]}$ \,\ $(b\in \{0,1\})$ \,\ 
\quad {\rm if} \quad $m \, \in \, \nnn_{\rm odd}$
\end{enumerate}
\end{enumerate}
\end{lemma}

\begin{proof} The claims 1) $\sim$ 4) follow immediately from
Note \ref{n3:note:2022-708a} and Note \ref{note:2022-707a}, 
using \eqref{n3:eqn:2022-708a1} and Lemma \ref{lemma:2022-702a}
and Lemma \ref{n3:lemma:2022-707a}. \,\ 
5) follows from 3) and 4) and \eqref{n3:eqn:2022-704b}.
\end{proof}

From this Lemma \ref{lemma:2022-708a} and Proposition 
\ref{prop:2022-201b}, we have the following:

\begin{lemma} \,\
\label{n3:lemma:2022-706b}
Let $f \in U^{[m,s]}$, then 
\begin{enumerate}
\item[{\rm 1)}] 
\begin{enumerate}
\item[{\rm (i)}] \, ${\rm ch}^{(+)}_{H(\Lambda^{[K(1), 0]})} 
\cdot f \, \in \, U^{[m+1,\frac12]}$
\quad if \quad $m \in \nnn$ and $s =\frac12$.
\item[{\rm (ii)}] \, ${\rm ch}^{(+)}_{H(\Lambda^{[K(1), 1]})} 
\cdot f \, \in \, U^{[m+1,s+\frac12]}$
\quad if \quad $m+2s \in \nnn_{\rm odd}$.
\end{enumerate}
\item[{\rm 2)}] 
\begin{enumerate}
\item[{\rm (i)}] \, ${\rm ch}^{(+)}_{H(\Lambda^{[K(2), 0]})} 
\cdot f \, \in \, U^{[m+2,s]}$ \quad if \,\ 
$(m,s) \in \nnn \times \{\frac12\} \,\ or \,\ 
\nnn_{\rm odd}\times \{0\}$.
\item[{\rm (ii)}] \, ${\rm ch}^{(+)}_{H(\Lambda^{[K(2), 2]})} 
\cdot f \, \in \, U^{[m+2,s]}$ \quad if \,\ 
$(m,s) \in \nnn \times \{\frac12\} \,\ or \,\ 
\nnn_{\rm odd}\times \{0\}$.
\item[{\rm (iii)}] \, ${\rm ch}^{(+)}_{H(\Lambda^{[K(2), 1]})} 
\cdot f \, \in \, U^{[m+2,s+\frac12]}$ \quad if \quad 
$m \in \nnn_{\rm odd}$.
\end{enumerate}
\item[{\rm 3)}] 
\begin{enumerate}
\item[{\rm (i)}] \, ${\rm ch}^{(+)}_{H(\Lambda^{[K(4), 1]})} 
\cdot f \, \in \, U^{[m+4,s+\frac12]}$ \quad if \quad 
$m \in \nnn_{\rm odd}$.
\item[{\rm (ii)}] \, ${\rm ch}^{(+)}_{H(\Lambda^{[K(4), 3]})} 
\cdot f \, \in \, U^{[m+4,s+\frac12]}$ \quad if \quad 
$m \in \nnn_{\rm odd}$.
\end{enumerate}
\end{enumerate}
\end{lemma}

In order to translate these properties to $V^{[m,s]}$, we extend the range 
of the parameter $s$ to $s \in \frac12 \zzz$ by putting 
$$
V^{[m,s]} \,\ := \,\ \left\{
\begin{array}{lcl}
V^{[m,\frac12]} & & {\rm if} \quad s \in \frac12+\zzz \\[1mm]
V^{[m,0]} & & {\rm if} \quad s \in \zzz
\end{array}\right.
$$
Then, by Lemma \ref{n3:lemma:2022-706a} and 
Lemma \ref{n3:lemma:2022-706b}, we obtain the following:

\begin{lemma} \,\
\label{n3:lemma:2022-706c}
Let $f \in V^{[m,s]}$, then 
\begin{enumerate}
\item[{\rm 1)}] 
\begin{enumerate}
\item[{\rm (i)}] \, ${\rm ch}^{(+)}_{H(\Lambda^{[K(1), 0]})} 
\cdot f \, \in \, V^{[m+1,\frac12]}$
\quad if \quad $m \in \nnn$ and $s =\frac12$.
\item[{\rm (ii)}] \, ${\rm ch}^{(+)}_{H(\Lambda^{[K(1), 1]})} 
\cdot f \, \in \, V^{[m+1,s+\frac12]}$
\quad if \quad $m+2s \in \nnn_{\rm odd}$.
\end{enumerate}
\item[{\rm 2)}] 
\begin{enumerate}
\item[{\rm (i)}] \, ${\rm ch}^{(+)}_{H(\Lambda^{[K(2), 0]})} 
\cdot f \, \in \, V^{[m+2,s]}$ \quad if \,\ 
$(m,s) \in \nnn \times \{\frac12\} \,\ or \,\ 
\nnn_{\rm odd}\times \{0\}$.
\item[{\rm (ii)}] \, ${\rm ch}^{(+)}_{H(\Lambda^{[K(2), 2]})} 
\cdot f \, \in \, V^{[m+2,s]}$ \quad if \,\ 
$(m,s) \in \nnn \times \{\frac12\} \,\ or \,\ 
\nnn_{\rm odd}\times \{0\}$.
\item[{\rm (iii)}] \, ${\rm ch}^{(+)}_{H(\Lambda^{[K(2), 1]})} 
\cdot f \, \in \, V^{[m+2,s+\frac12]}$ \quad if \quad 
$m \in \nnn_{\rm odd}$.
\end{enumerate}
\item[{\rm 3)}] 
\begin{enumerate}
\item[{\rm (i)}] \, ${\rm ch}^{(+)}_{H(\Lambda^{[K(4), 1]})} 
\cdot f \, \in \, V^{[m+4,s+\frac12]}$ \quad if \quad 
$m \in \nnn_{\rm odd}$.
\item[{\rm (ii)}] \, ${\rm ch}^{(+)}_{H(\Lambda^{[K(4), 3]})} 
\cdot f \, \in \, V^{[m+4,s+\frac12]}$ \quad if \quad 
$m \in \nnn_{\rm odd}$.
\end{enumerate}
\end{enumerate}
\end{lemma}

Now we consider the $\ccc((q^{\frac12}))$-space of N=3 characters:
{\allowdisplaybreaks
\begin{eqnarray*}
\overset{N=3}{CH}{}^{(+)[K(m),j]} 
&=&
\ccc((q^{\frac12}))\text{-linear span of} \,\ \left\{ 
{\rm ch}^{(+)}_{H(\Lambda^{[K(m), m_2]})}(\tau,z)
\,\ ; \,\ 
\begin{array}{l}
m_2 \in j+2\zzz \\[1mm]
0 \leq m_2 \leq m
\end{array} \right\} 
\\[2mm]
&:=& 
\left\{
\frac{f}{\overset{N=3}{R}{}^{(+)}} \quad ; \quad 
f \in V^{[m,\frac{j+1}{2}]} \right\}
\end{eqnarray*}}
for $m \in \nnn$ and $j \in \zzz$. Note that 
$$
\overset{N=3}{CH}{}^{(+)[K(m),j]} \,\ = \,\ \left\{
\begin{array}{lcl}
\overset{N=3}{CH}{}^{(+)[m,0]} & & {\rm if} \quad j \in \zzz_{\rm even} \\[1mm]
\overset{N=3}{CH}{}^{(+)[m,1]} & & {\rm if} \quad j \in \zzz_{\rm odd}
\end{array}\right.
$$

Then, in terms of $\overset{N=3}{CH}{}^{(+)[K(m),j]}$, 
Lemma \ref{n3:lemma:2022-706c} is restated as follows:

\begin{prop} \,\
\label{n3:prop:2022-706a}
Let $f \in \overset{N=3}{CH}{}^{(+)[K(m),j]}$, then 
\begin{enumerate}
\item[{\rm 1)}] 
\begin{enumerate}
\item[{\rm (i)}] \, ${\rm ch}^{(+)}_{H(\Lambda^{[K(1), 0]})} 
\cdot f \, \in \, \overset{N=3}{CH}{}^{(+)[K(m+1), \, 0]}$
\quad if \quad $m \in \nnn$ and $j=0$
\item[{\rm (ii)}] \, ${\rm ch}^{(+)}_{H(\Lambda^{[K(1), 1]})} 
\cdot f \, \in \, \overset{N=3}{CH}{}^{(+)[K(m+1), \, j+1]}$
\quad if \quad $m+j \in \zzz_{\rm even}$
\end{enumerate}
\item[{\rm 2)}] 
\begin{enumerate}
\item[{\rm (i)}] \, ${\rm ch}^{(+)}_{H(\Lambda^{[K(2), 0]})} 
\cdot f \, \in \, \overset{N=3}{CH}{}^{(+)[K(m+2), \, j]}$ \quad if \,\ 
$(m,j) \in \nnn \times \{0\} \,\ or \,\ 
\nnn_{\rm odd}\times \{1\}$.
\item[{\rm (ii)}] \, ${\rm ch}^{(+)}_{H(\Lambda^{[K(2), 2]})} 
\cdot f \, \in \, \overset{N=3}{CH}{}^{(+)[K(m+2), \, j]}$ \quad if \,\ 
$(m,j) \in \nnn \times \{0\} \,\ or \,\ 
\nnn_{\rm odd}\times \{1\}$.
\item[{\rm (iii)}] \, ${\rm ch}^{(+)}_{H(\Lambda^{[K(2), 1]})} 
\cdot f \, \in \, \overset{N=3}{CH}{}^{(+)[K(m+2), \, j+1]}$ \quad if \quad 
$m \in \nnn_{\rm odd}$.
\end{enumerate}
\item[{\rm 3)}] 
\begin{enumerate}
\item[{\rm (i)}] \, ${\rm ch}^{(+)}_{H(\Lambda^{[K(4), 1]})} 
\cdot f \, \in \, \overset{N=3}{CH}{}^{(+)[K(m+4), \, j+1]}$ \quad if \quad 
$m \in \nnn_{\rm odd}$.
\item[{\rm (ii)}] \, ${\rm ch}^{(+)}_{H(\Lambda^{[K(4), 3]})} 
\cdot f \, \in \, \overset{N=3}{CH}{}^{(+)[K(m+4), \, j+1]}$ \quad if \quad 
$m \in \nnn_{\rm odd}$.
\end{enumerate}
\end{enumerate}
\end{prop}

We expect the following:

\begin{conj} 
\label{n3:conj:2022-706a}
The following formula will hold for all $m,m' \in \nnn$ and $j, j' \in \zzz$:
\begin{subequations}
\begin{equation}
\overset{N=3}{CH}{}^{(+)[K(m),j]} \cdot \overset{N=3}{CH}{}^{(+)[K(m'),j']}
\,\ \subset \,\ 
\overset{N=3}{CH}{}^{(+)[K(m+m'), \, j+j']}
\label{n3:eqn:2022-706a}
\end{equation}
Namely, for $m,m' \in \nnn$ and $m_2, m_2' \in \zzz$ satisfying 
$0 \leq m_2 \leq m$ and $0 \leq m_2' \leq m'$, 
there will exist functions $b^{(m,m_2),(m',m_2')}_{m_2''} \in \ccc((q^{\frac12}))$ 
such that 
\begin{equation}
{\rm ch}^{(+)}_{H(\Lambda^{[K(m), m_2]})}(\tau, z) \cdot 
{\rm ch}^{(+)}_{H(\Lambda^{[K(m'), m_2']})}(\tau, z) 
= \hspace{-7mm}
\sum_{\substack{\\[0.5mm] m_2'' \, \in \, m_2+m_2'+2\zzz \\[1mm] 
0 \leq m_2'' \leq m+m'}} 
\hspace{-7mm}
b^{(m,m_2),(m',m_2')}_{m_2''}(\tau) \cdot
{\rm ch}^{(+)}_{H(\Lambda^{[K(m+m'), m_2'']})}(\tau, z) 
\label{n3:eqn:2022-706b}
\end{equation}
\end{subequations}
\end{conj}

Then, by Propositions \ref{prop:2022-607a} and \ref{prop:2022-622a} 
and \ref{n3:prop:2022-706a}, we have the following:

\begin{prop} 
\label{n3:thm:2022-706a}
The formula \eqref{n3:eqn:2022-706a} in Conjecture \ref{n3:conj:2022-706a} 
is true in the following cases:
\begin{enumerate}
\item[{\rm 1)}] \,\ $(m,m_2)=(1,0)$ \,\ and \,\ 
$(m', m_2') \in \nnn \times \zzz_{\rm even}$
\item[{\rm 2)}] \,\ $(m,m_2)=(1,1)$ \,\ and \,\ 
$m'+ m_2' \in \zzz_{\rm even}$
\item[{\rm 3)}] \,\ $(m,m_2)=(2,0)$ or $(2,2)$
\,\ and \,\ $(m', m_2') \in \nnn \times \zzz_{\rm even}$
\item[{\rm 4)}] \,\ $(m,m_2)=(2,0)$ or $(2,2)$
\,\ and \,\ $(m', m_2') \in \nnn_{\rm odd} \times \zzz_{\rm odd}$
\item[{\rm 5)}] \,\ $(m,m_2)=(2,1)$ \,\ and \,\ 
$m' \in \nnn_{\rm odd}$
\item[{\rm 6)}] \,\ $(m,m_2)=(4,1)$ or $(4,3)$ \,\ and \,\ 
$m' \in \nnn_{\rm odd}$
\item[{\rm 7)}] \,\ $m=m'=1$
\item[{\rm 8)}] \,\ $m=m'=2$ \,\ and \,\ $m_2+m_2' \in \zzz_{\rm odd}$
\end{enumerate}
\end{prop}

\end{document}